\documentclass[12pt]{amsart}

\usepackage{geometry}
\usepackage[utf8]{inputenc}
\usepackage[english]{babel}
\usepackage{amssymb}
\usepackage{amsmath}
\usepackage{amsthm}
\usepackage{amscd}
\usepackage{enumerate}
\usepackage{graphicx}
\usepackage{tikz}

\newcommand{\eps}{\varepsilon}
\newcommand{\vect}{\text{span}}
\newcommand{\gen}{\text{Gen}}

\DeclareMathOperator{\prt}{par} 
\DeclareMathOperator{\Chi}{Chi}
\DeclareMathOperator{\Orb}{Orb}
\newcommand{\rt}{{\tt{root}}}

\def\RR{\mathbb R}
\def\NN{\mathbb N}
\def\ZZ{\mathbb Z}
\def\CC{\mathbb C}
\def\KK{\mathbb K}

\theoremstyle{plain}
\newtheorem{theorem}{Theorem}[section]
\newtheorem{lemma}[theorem]{Lemma}
\newtheorem{corollary}[theorem]{Corollary}
\newtheorem{proposition}[theorem]{Proposition}

\theoremstyle{definition}

\newtheorem{example}[theorem]{Example}

\theoremstyle{remark}
\newtheorem{remark}[theorem]{Remark}
\newtheorem{problem}[theorem]{Problem}

\numberwithin{equation}{section}

\title{Dynamics of weighted shifts on directed trees}
\author{Karl-G. Grosse-Erdmann and Dimitris Papathanasiou}
\address{Karl-G. Grosse-Erdmann, 
Département de Math\'ematique, Universit\'e de Mons, 20 Place du Parc, 7000 Mons, Belgium}
\email{kg.grosse-erdmann@umons.ac.be}
\address{Dimitris Papathanasiou, Euler International Mathematical Institute, Pesochnaya Naberezhnaya 10, Saint Peterburg, Russia, 197376.}
\email{dpapath@bgsu.edu}


\begin{document}

\begin{abstract}
We study the dynamical behaviour of weighted shifts defined on sequence spaces of a directed tree. In particular, we characterize their boundedness as well as when they are hypercyclic, weakly mixing and mixing.
\end{abstract}

\maketitle

\section{Introduction}

\subsection{Weighted shifts in linear dynamics} As Héctor N. Salas famously said, ``the class of weighted shifts is a favorite testing ground for operator-theorists'', \cite{Salas}. This is particularly true in linear dynamics, where shift operators play a central, and sometimes crucial, role. The reason is that their simple definition allows to compute explicitly the full orbit of a vector, while the class is rather flexible since one may choose both the weights and the underlying space. As a consequence, weighted shifts provide important examples and counter-examples in linear dynamics like a hypercyclic operator whose adjoint is also hypercyclic \cite{Sal91}, or a frequently hypercyclic operator that is not chaotic \cite{BaGr07}. They are also a key element in constructing hypercyclic operators on arbitrary infinite-dimensional separable Banach spaces \cite{Ans97}, \cite{Ber99}. In addition, some natural operators on spaces of analytic functions like the differentiation operator $D$, the multiplication operator $M_z$, or the Bergman backward shift are weighted shift operators, see \cite{Shi74}, \cite{BoSh00}, \cite[Sections 4.1, 4.4]{GrPe11}.

On the other hand, some of the counter-examples provided by weighted shifts are rather complicated. For example, one may wonder whether there is not a more elementary operator than the one constructed in \cite{BaRu15} (see also \cite{BoGE18}) that is upper frequently hypercyclic but not frequently hypercyclic. In some cases, weighted shifts fail completely: Menet's recent deep construction of a chaotic operator that is not frequently hypercyclic \cite{Men17} could not have been done with a weighted backward shift, see \cite[Corollary 9.14]{GrPe11}.

It is thus of great practical and theoretical interest to enlarge the class of weighted shift operators. Now, these operators are not only determined by their weights and their underlying space, but also by the set of indices for the sequences they act upon: traditionally, these are the set $\NN$ of natural numbers for the unilateral shifts, and the set $\ZZ$ of integers for the bilateral shifts. So, a natural way to extend the notion of weighted shifts is to replace these sets by directed trees.

\subsection{Linear dynamics and universality on trees} Analysis on trees goes back at least to the work of Cartier \cite{Car72}, \cite{Car73}, who introduces and studies harmonic functions on trees. His trees are not directed, but by singling out, as Cartier does, a particular vertex, one may consider them as rooted directed trees, see also \cite{ANP17}, \cite{MaRi20}. Moreover, he demands that each vertex has only a finite number of children. These trees can be regarded as discrete analogues of the complex unit disk; Cartier also defines the boundary of a tree, which is then the analogue of the complex unit circle. Hardy spaces on trees were subsequently introduced by Kor\'anyi, Picardello and Taibleson \cite{KPT87}. Motivated by work of Bourdon and Shapiro \cite{BoSh90} on the hypercyclicity of composition operators on the Hardy space, Pavone \cite{Pav92} studies hypercyclicity of composition operators on $L^p$-spaces of the boundary of a Cartier tree. In a different direction, Abakumov, Nestoridis and Picardello \cite{ANP17}, \cite{ANP21} obtain a universality result on harmonic functions on trees that is analogous to Menshov's universal trigonometric series \cite{Men45}, \cite{Men47}.

There is a growing interest in the analysis on trees, see, for example, \cite{GaSo05} and \cite{CoMa17} and the literature cited therein.
 
\subsection{Weighted shifts on trees} In this paper we study the dynamical behaviour of shift operators on trees. Perhaps surprisingly, weighted shifts on arbitrary directed trees were only introduced and studied quite recently in the work of Jab{\l}o\'nski, Jung and Stochel \cite{JJS12}; these authors provide a detailed study of their operator theoretic properties on the $\ell^2$-space of the tree. Martínez-Avenda\~{n}o \cite{Mar17} then initiated the study of the dynamical properties of these operators on arbitrary $\ell^p$-spaces; Martínez-Avenda\~{n}o and Rivera-Guasco \cite{MaRi20} consider equally the unweighted shift on the Lipschitz space of the tree. 

In the classical case of the trees $\NN$ and $\ZZ$, Salas \cite{Salas} characterized hypercyclicity and weak mixing on the $\ell^p$-spaces, while Costakis and Sambarino \cite{CoSa} obtained the corresponding result for the mixing property. It is the aim of this paper to solve completely the problem of characterizing when a weighted shift on a directed tree is hypercyclic, weakly mixing or mixing, provided that the underlying space is of type $\ell^p$ or $c_0$. We thereby complete the study started by Martínez-Avenda\~{n}o \cite{Mar17}.
 
Let us add a word of caution. While Jab{\l}o\'nski et al.\ \cite{JJS12} consider weighted shifts on the (unweighted) $\ell^2$-space, Martínez-Avenda\~{n}o \cite{Mar17} has chosen to study (unweighted) shifts on weighted $\ell^p$-spaces. It is well known, as explained in \cite{Shi74}, \cite[Section 4.1]{GrPe11} for the trees $\NN$ and $\ZZ$ and mentioned in \cite{Mar17} for arbitrary trees, that these are merely two different representations of the ``same operator''. Motivated by the classical theory, we will consider here both representations in parallel. However, in order to make the transition between the two formulations more transparent we will employ a notation that differs somewhat from that of Martínez-Avenda\~{n}o, as will be explained in Section \ref{s-shift}.

\subsection{Linear dynamics} We proceed by providing the basic notions from linear dynamics that we will consider. For further details, in particular for the explanation of some additional notions mentioned above, we refer the reader to the books \cite{BM09} and \cite{GrPe11}. A bounded (linear) operator $T$ on a Banach space $X$ is called \textit{hypercyclic} provided there is some $x\in X$ whose orbit under $T$,
\[
\Orb(x,T)=\{x, Tx, T^2x,\ldots \},
\]
is dense in $X$. If furthermore the operator defined by
\[
(T\oplus T)(x,y)=(Tx, Ty), \quad (x,y)\in X\times X
\]
is hypercyclic, then $T$ is called \textit{weakly mixing}. In the context of separable Banach spaces, hypercyclicity is equivalent to \textit{topological transitivity} which requires that for every pair $U$ and $V$ of non-empty open subsets of $X$ there is some $n\in \NN$ such that $T^n(U)\cap V \neq \varnothing$. If in addition, $T^k(U)\cap V\neq \varnothing$ for all $k\geq n$, then $T$ will be called \textit{mixing}. The main tool for proving the weak mixing and mixing properties of an operator is the following, known as the \textit{Hypercyclicity Criterion}, see \cite[Theorem 3.12]{GrPe11}; note, however, that there are hypercyclic operators that are not weakly mixing, see \cite{DeRe09}, \cite{BaMa07}.

\begin{theorem}[Hypercyclicity Criterion]\label{t-hypcrit}
Let $X$ be a separable Banach space and $T$ a bounded operator on $X$. Assume that there exist dense subsets $X_0, Y_0$ of $X$, an increasing sequence $(n_k)_k$ of positive integers, and maps $R_{n_k}:Y_0\rightarrow X$, $k\geq 1$, such that, for any $x\in X_0$ and $y\in Y_0$,
\begin{enumerate}[{\rm (i)}]
    \item $T^{n_k}x\rightarrow 0$,
    \item $R_{n_k}y\rightarrow 0$, and
    \item $T^{n_k}R_{n_k}y\rightarrow y$
\end{enumerate}
as $k\to\infty$. Then $T$ is weakly mixing. 

If the conditions hold for the full sequence $(n_k)_k=(n)_n$ then $T$ is mixing.
\end{theorem}

We mention that we will also need a variant of the criterion, see Proposition \ref{p-criterion} below.

\section{Shifts on trees}\label{s-shift}

\subsection{Directed trees}\label{subs-dirtr}

Let $(V,E)$ be a \textit{directed tree}, that is, a connected directed graph consisting of a countable set $V$ of vertices and a set $E\subset V\times V\setminus \{(v,v):v\in V\}$ of edges that has no cycles and for which each vertex $v\in V$ has at most one $w\in V$ such that $(w,v)\in E$. The terminology that we use is standard. Thus we can say that $w$ is the \textit{parent} of $v$, and the \textit{indegree} of $v$ is either 0 if it does not have a parent or 1 if it does. For details we refer to \cite{Car72}, \cite{GYZ14}, \cite{JJS12} and \cite{Mar17}; but note that in \cite{Car72} and \cite{GYZ14} trees are not directed. We only demand connectedness in order to avoid that the tree splits into a disjoint union of unrelated trees.

A directed tree has at most one vertex without a parent, called the \textit{root} of the tree and denoted $\rt$. The parent of a vertex $v\neq \rt$ is denoted by $\prt(v)$. Any vertex whose parent is $v$ is called a child of $v$, and the set of children of $v$ is denoted by $\Chi(v)$. More generally, for any $n\geq 1$,
\[
\Chi^n(v) = \Chi(\Chi(\ldots(\Chi(v)))),\quad \text{($n$ times)},
\]
with $\Chi^0(v)=\{v\}$. The \textit{outdegree} of a vertex $v$ is $|\Chi(v)|$, the cardinality of $\Chi(v)$. A \textit{leaf} is a vertex with outdegree $0$.

For two vertices $u,v\in V$, let us write $u\sim v$ if there is some $n\in \mathbb{N}_0$ such that $\prt^n(u)=\prt^n(v)$, which defines an equivalence relation on $V$. For any $v\in V$, the equivalence class $[v]$ of $v$ is called the \textit{generation} of $v$. Let $v_0 \in V$ be a fixed vertex; if the directed tree has a root, we pick $v_0=\rt$. We will enumerate the generations with respect to $v_0$ by defining
\[
\gen_n=\Chi^n(\rt), \quad n\in \NN_0,
\]
if the directed tree has a root, and
\[
\gen_n=\gen_n(v_0)= \{v\in V : \exists m\geq 0, \prt^{n+m}(v)=\prt^{m}(v_0)\},\quad n\in \ZZ,
\] 
otherwise. If there is an $n\in \mathbb{N}_0$ such that $\prt^n(v)=\prt^n(u)$ for each $u\in [v]$ then $[v]$ has a \textit{common ancestor}. We will say that an unrooted tree has a \textit{free left end} if there is some $n\in\ZZ$ such that $\gen_k$ is a singleton for all $k\leq n$. In that case, each generation has a common ancestor. We talk of a free \textit{left} end because we imagine trees to be directed from left to right, see Figure \ref{f-leftend}.

\begin{figure}[h!]%
\begin{tikzpicture}[scale=1]
\draw[->,>=latex,dashed] (0,0) - - (1,0);
\draw[->,>=latex] (1,0) - - (2,0);
\draw[->,>=latex] (2,0) - - (3,0);
\draw[->,>=latex] (3,0) - - (4,0);
\draw[->,>=latex] (4,0) - - (5,.6);\draw[fill] (5,.6) circle (1pt);
\draw[->,>=latex] (4,0) - - (5,-.6);\draw[fill] (5,-.6) circle (1pt);
\draw[->,>=latex] (4,0) - - (5,0);\draw[fill] (5,0) circle (1pt);
\draw[->,>=latex] (5,.6) - - (6,1);\draw[fill] (6,1) circle (1pt);
\draw[->,>=latex] (5,.6) - - (6,.7);\draw[fill] (6,.7) circle (1pt);
\draw[->,>=latex] (5,.6) - - (6,0.4);\draw[fill] (6,.4) circle (1pt);
\draw[->,>=latex] (5,0) - - (6,0);\draw[fill] (6,0) circle (1pt);
\draw[->,>=latex] (5,-.6) - - (6,-.4);\draw[fill] (6,-.4) circle (1pt);
\draw[->,>=latex] (5,-.6) - - (6,-.8);\draw[fill] (6,-.8) circle (1pt);
\draw[dashed] (6,1) - - (7,1.2);
\draw[dashed] (6,.7) - - (7,.7);
\draw[dashed] (6,.4) - - (7,.3);
\draw[dashed] (6,0) - - (7,0);
\draw[dashed] (6,-.4) - - (7,-.3);
\draw[dashed] (6,-.8) - - (7,-.9);
\draw[fill] (1,0) circle (1pt);
\draw[fill] (2,0) circle (1pt);
\draw[fill] (3,0) circle (1pt);
\draw[fill] (4,0) circle (1pt);
\end{tikzpicture}
\caption{An unrooted directed tree with a free left end}%
\label{f-leftend}%
\end{figure}
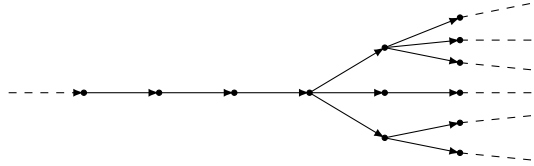

A directed tree $(V,E)$ will be called \textit{symmetric} if 
\[
|\Chi(u)|=|\Chi(v)|
\]
whenever $u\sim v$. This is equivalent to demanding that the map
\[
|\Chi(\cdot)|: V\rightarrow \NN_0\cup \{ \infty \}
\]
depends only on the generation. Let us fix $v_0\in V$, with $v_0=\rt$ in the rooted case, and let us enumerate the generations with respect to $v_0$. For $v\in \gen_n$, we set 
\[
\gamma_n=|\Chi(v)|, 
\]
which is well defined if $(V,E)$ is symmetric.

\subsection{Sequence spaces} In this paper we restrict our attention to spaces of type $\ell^p$ and $c_0$. We recall here the definition of  weighted versions of these spaces on general index sets. Thus, let $V$ be an arbitrary countable set; a possible graph structure on $V$ will not play any role in this subsection. Let $\mu=(\mu_v)_{v\in V}\in \KK^V$ be a family of non-zero (real or complex) numbers, called a \textit{weight}. For $1\leq p < \infty$, the \textit{weighted $\ell^p$-space} of $V$ is defined as
\[
\ell^p(V,\mu) = \Big\{ f\in \KK^V : \|f\|^p:= \sum_{v\in V} |f(v)\mu_v|^p <\infty\Big\}.
\]
Note that, unlike in \cite{Mar17}, the weights also carry a power $p$; in addition, it will be convenient for us not to restrict to positive weights. As usual, $\KK$ denotes either $\RR$ or $\CC$.

The space $\ell^\infty(V,\mu)$ is defined similarly. The only reason for considering this space in our paper is that it is the dual of some $\ell^1$-space; from a dynamical point of view the space is of little interest because it is non-separable (unless $V$ is finite). 

In addition we will consider the \textit{weighted $c_0$-space} of $V$,
\[
c_0(V,\mu)=\{ f\in \KK^V: \forall \eps >0, \exists F\subset V \mbox{ finite}, \forall v\in V\setminus F,  |f(v)\mu_v|<\eps\}, 
\]
normed by
\[
\|f\|=\sup_{v\in V}|f(v)\mu_v|;
\]
it is a closed subspace of $\ell^\infty(V,\mu)$.

Setting $e_v=\chi_{\{v\}}$ for $v\in V$ we observe that $\vect\{e_v: v\in V\}$ is dense in $\ell^p(V,\mu)$, $1\leq p<\infty$, and in $c_0(V,\mu)$. For any element $f\in \KK^V$, we call $\{v\in V : f(v)\neq 0\}$ its \textit{support}.

If $\mu=1$ for all $v\in V$, the corresponding unweighted spaces are denoted by
\[
\ell^p(V), \, 1\leq p\leq \infty,\quad\text{and}\quad c_0(V).
\]

Under the natural pairing
\begin{equation}\label{eq-dual}
\langle f,g \rangle =\sum_{v\in V} f(v)\overline{g(v)}
\end{equation}
we have the following dual spaces:
\[
\ell^p(V,\mu)^* = \ell^{p^*}\!(V,1/\overline{\mu}), \, 1\leq p <\infty,
\]
where $p^*$ is the conjugate exponent, and 
\[
c_0(V,\mu)^* = \ell^1(V,1/\overline{\mu}).
\]

\subsection{Shift operators on directed trees}

Throughout this subsection, let $(V,E)$ be a directed tree. We will consider weighted (forward) shifts and weighted backward shifts on our spaces. Let $\lambda=(\lambda_v)_{v\in V}$ be a family of non-zero scalars, that is, a weight. 

According to \cite{JJS12}, the \textit{weighted (forward) shift} $S_\lambda$ is defined on $\KK^V$ by
\[
(S_{\lambda}f)(v)=
\begin{cases}
\lambda_v f(\prt(v)), & \mbox{if } v\neq \rt,\\
0, & \mbox{if }   v=\rt,
\end{cases}
\]
where $v\in V$; note that the value of $\lambda_{\rt}$ is not used.

Similarly, by \cite{Mar17}, the \textit{weighted backward shift} $B_\lambda$ is defined formally on $\KK^V$ by
\[
(B_{\lambda}f)(v)=\sum_{u\in \Chi(v)}\lambda_uf(u), \quad v\in V,
\]
where, as usual, an empty sum is zero. When one considers $B_\lambda$ on a specific space $X$ one needs to ensure that all the series defining $B_\lambda f$ converge for any $f\in X$. 

If $\lambda=1$ for all $v\in V$, the corresponding unweighted shifts are denoted by $S$ and $B$.

There is an important link between the two types of shifts: backward shifts are the adjoints of forward shifts. In fact, this property has motivated the definition of backward shifts in \cite{Mar17}. However, the notation in that paper obscures the following simple fact.

\begin{proposition}
Let $\lambda$ and $\mu$ be weights.
\begin{enumerate}[{\rm (a)}]
\item If $S_\lambda$ is a bounded operator on $\ell^p(V,\mu)$, $1\leq p<\infty$, or on $c_0(V,\mu)$, then $S_\lambda^\ast=B_{\overline{\lambda}}$ on $\ell^{p^\ast}\!(V,1/\overline{\mu})$ or on $\ell^{1}(V,1/\overline{\mu})$, respectively.
\item If $B_\lambda$ is a bounded operator on $\ell^p(V,\mu)$, $1\leq p<\infty$, or on $c_0(V,\mu)$, then $B_\lambda^*=S_{\overline{\lambda}}$ on $\ell^{p^*}\!(V,1/\overline{\mu})$ or on $\ell^{1}(V,1/\overline{\mu})$, respectively.
\end{enumerate}
\end{proposition}

\begin{proof} It is readily verified that, with respect to the dual pairing \eqref{eq-dual},
\[
\langle S_\lambda f,g\rangle = \langle  f,B_{\overline{\lambda}} g\rangle\quad \text{and}\quad 
\langle B_\lambda f,g\rangle = \langle  f,S_{\overline{\lambda}} g\rangle,
\]
where $f$ is in the underlying space and $g$ in its dual; the rearrangements needed in these calculations can be justified by first treating $|\lambda|$, $|f|$, and $|g|$.
\end{proof}

We now characterize when a (forward or backward) shift is a bounded operator. The closed graph theorem shows that for this it suffices that the operator in question is a well-defined map from the space back into the space.

The following result generalizes Jab{\l}o\'nski et al.\ \cite[Proposition 3.1.8]{JJS12} and Martínez-Aven\-da\~{n}o \cite[Proposition 3.2]{Mar17}.

\begin{proposition}
Let $\lambda$ and $\mu$ be weights. 

{\rm (a)} Let $1\leq p < \infty$. Then the following assertions are equivalent:
\begin{enumerate}[\rm (i)]
    \item $S_\lambda$ is a bounded operator on $\ell^p(V,\mu)$;
    \item there is a constant $M>0$ such that $\sum_{u\in \Chi(v)}|\lambda_u\mu_u|^p\leq M |\mu_v|^p$, $v\in V$.
\end{enumerate}
In this case, $\|S_\lambda\|=\sup_{v\in V}\frac{1}{|\mu_v|}(\sum_{u\in \Chi(v)}|\lambda_u\mu_u|^p)^{1/p}$.

{\rm (b)} The following assertions are equivalent:
\begin{enumerate}[\rm (i)]
    \item $S_\lambda$ is a bounded operator on $\ell^\infty(V,\mu)$;
    \item there is a constant $M>0$ such that $|\lambda_v\mu_v|\leq M |\mu_{\prt(v)}|$, $v\in V\setminus\{\rt\}$.
\end{enumerate}
In this case, $\|S_{\lambda}\|=\sup_{v\neq\rt}\big|\frac{\lambda_v\mu_v}{\mu_{\prt(v)}}\big|$.

{\rm (c)} The following assertions are equivalent:
\begin{enumerate}[\rm (i)]
    \item $S_\lambda $ is a bounded operator on $c_0(V,\mu)$;
    \item there is a constant $M>0$ such that $|\lambda_v\mu_v|\leq M |\mu_{\prt(v)}|$, $v\in V\setminus\{\rt\}$, and for each $v\in V$, $(\lambda_u\mu_u)_{u\in \Chi(v)}\in c_0(\Chi(v))$. 
\end{enumerate}
In this case, $\|S_\lambda \|=\sup_{v\neq\rt}\big|\frac{\lambda_v\mu_v}{\mu_{\prt(v)}}\big|$.
\end{proposition}

\begin{proof} Parts (a) and (b) follow immediately from the fact that
\[
\|S_\lambda f\|_{\ell^p(V,\mu)}^p = \sum_{v\neq \rt} |\lambda_v f(\prt(v))\mu_v|^p =  \sum_{w\in V} |f(w)\mu_w|^p \Big(\frac{1}{|\mu_w|^p}\sum_{v\in\Chi(w)}|\lambda_v\mu_v|^p\Big)
\]
and
\begin{equation}\label{eq-linf}
\|S_\lambda f\|_{\ell^\infty(V,\mu)} = \sup_{v\neq \rt} |\lambda_v f(\prt(v))\mu_v| =  \sup_{w\in V} |f(w)\mu_w| \Big(\frac{1}{|\mu_w|}\sup_{v\in\Chi(w)}|\lambda_v\mu_v|\Big).
\end{equation}

As for (c), applying \eqref{eq-linf} to $e_v$, $v\in V$, and noting that $S_\lambda e_v\in c_0(V,\mu)$ shows that (i) implies (ii). Conversely, if (ii) holds then, by (b), $S_\lambda$ is a bounded operator on $\ell^\infty(V,\mu)$; moreover, by the second part of (ii), $S_\lambda e_v\in c_0(V,\mu)$ for all $v\in V$. Since $\vect\{e_v:v\in V\}$ is dense in $c_0(V,\mu)$ and $c_0(V,\mu)$ is a closed subspace of $\ell^\infty(V,\mu)$, (i) follows.
\end{proof}

We consider the analogue for backward shift operators, which improves, in particular, \cite[Proposition 4.5]{Mar17}; see also \cite[Proposition 3.4.1]{JJS12}.

\begin{proposition}\label{p-Bbounded}
Let $\lambda$ and $\mu$ be weights. 

{\rm (a)} The following assertions are equivalent:
\begin{enumerate}[\rm (i)]
    \item $B_\lambda$ is a bounded operator on $\ell^1(V,\mu)$;
    \item $\sup_{v\in V\setminus \{\rt\}} |\mu_{\prt(v)}|\big|\frac{\lambda_v}{\mu_v}\big| < \infty$.
\end{enumerate}
In this case, $\|B_\lambda\|=\sup_{v\in V\setminus \{\rt\}} |\mu_{\prt(v)}|\big|\frac{\lambda_v}{\mu_v}\big|$.

{\rm (b)} Let $1< p < \infty$. Then the following assertions are equivalent:
\begin{enumerate}[\rm (i)]
    \item $B_\lambda$ is a bounded operator on $\ell^p(V,\mu)$;
    \item $\sup_{v\in V} |\mu_v|^{p^*}\sum_{u\in \Chi(v)}\big|\frac{\lambda_u}{\mu_u}\big|^{p^*} < \infty$.
\end{enumerate}
In this case, $\|B_\lambda\|=\sup_{v\in V}|\mu_v|(\sum_{u\in \Chi(v)}\big|\frac{\lambda_u}{\mu_u}\big|^{p^*})^{1/{p^*}}$.

{\rm (c)} The following assertions are equivalent:
\begin{enumerate}[\rm (i)]
    \item $B_\lambda$ is a bounded operator on $\ell^\infty(V,\mu)$;
		\item $B_\lambda$ is a bounded operator on $c_0(V,\mu)$;
    \item $\sup_{v\in V}|\mu_v|\sum_{u\in \Chi(v)}\big|\frac{\lambda_u}{\mu_u}\big|< \infty$.  
\end{enumerate}
In this case, $\|B_\lambda\|=\sup_{v\in V}|\mu_v|\sum_{u\in \Chi(v)}\big|\frac{\lambda_u}{\mu_u}\big|$ as an operator on $\ell^\infty(V,\mu)$ and on $c_0(V,\mu)$.
\end{proposition}

\begin{proof}
(b) (ii) $\Rightarrow$ (i). Let $M$ denote the supremum in (ii). Using H\"older's inequality we obtain that
\begin{align*}
\sum_{v\in V}\Big|\sum_{u\in \Chi(v)}|\lambda_uf(u)|\, \mu_v\Big|^p &\leq \sum_{v\in V}|\mu_v|^p\Big(\sum_{u\in \Chi(v)}\Big|\frac{\lambda_u}{\mu_u}\Big|^{p^*}\Big)^{p/p^*}\Big(\sum_{u\in \Chi(v)}|f(u)\mu_u|^{p}\Big)\\
&\leq M^{p/p^*}\sum_{v\in V}\sum_{u\in \Chi(v)}|f(u)\mu_u|^{p}\leq M^{p/p^*} \|f\|^p_{\ell^p(V,\mu)}.
\end{align*}
Thus, we get that $B_{\lambda}$ is well-defined and 
\begin{equation}\label{eq:3}
\|B_{\lambda}\| \leq M^{1/p^*}.
\end{equation}

(i) $\Rightarrow$ (ii). We consider $v\in V$. Let $I:\ell^p(\Chi(v),\mu \chi_{\Chi(v)})\to \ell^p(V,\mu)$ be the canonical isometric embedding. We define 
\[
\phi_v(f):=\mu_v(B_{\lambda}If)(v)=\sum_{u\in \Chi(v)}\mu_v\lambda_uf(u), \quad f\in \ell^p(\Chi(v),\mu \chi_{\Chi(v)}),
\]
which is a linear functional on $\ell^p(\Chi(v),\mu \chi_{\Chi(v)})$. Since
\[
|\phi_v(f)|=\|B_{\lambda}If\|_{\ell^p(V,\mu)}\leq \|B_{\lambda}\| \|If\|_{\ell^p(V,\mu)}= \|B_{\lambda}\| \|f\|_{\ell^p(\Chi(v),\mu \chi_{\Chi(v)})},
\]
we have that $\phi_v\in \ell^p(\Chi(v),\mu \chi_{\Chi(v)})^{\ast}=\ell^{p^*}\!(\Chi(v),1/\overline{\mu}\, \chi_{\Chi(v)})$ with $\| \phi_v\| \leq \|B_{\lambda}\|$. In other words,  
\[
\Big( \sum_{u\in\Chi(v)} \Big|\frac{\mu_v\lambda_u}{\overline{\mu}_u}\Big|^{p^*}\Big)^{1/p^*} \leq \|B_{\lambda}\|.
\]
This implies (ii), and together with \eqref{eq:3} we have also determined $\|B_{\lambda}\|$.

The proof of (a) is essentially the same, noting that $p^*=\infty$.

(c) (i) $\Rightarrow$ (ii). Let $B_\lambda$ be a bounded operator on $\ell^\infty(V,\mu)$. Note that $B_\lambda e_v\in c_0(V,\mu)$ for all $v\in V$; since $\vect\{e_v:v\in V\}$ is dense in $c_0(V,\mu)$ and $c_0(V,\mu)$ is a closed subspace of $\ell^\infty(V,\mu)$, (ii) follows.

(ii) $\Rightarrow$ (iii). Let $v\in V$. It follows as in (b), implication (i) $\Rightarrow$ (ii), that
\begin{equation}\label{eq:4}
\sup_{v\in V}\sum_{u\in\Chi(v)} \Big|\frac{\mu_v\lambda_u}{\mu_u}\Big| \leq \|B_{\lambda}\|_{c_0(V,\mu)\to c_0(V,\mu)},
\end{equation}
hence (iii).

(iii) $\Rightarrow$ (i). Let $M$ denote the supremum in (iii). Then we have for $f\in\ell^\infty(V,\mu)$
\begin{align*}
\sup_{v\in V}\sum_{u\in \Chi(v)}|\lambda_uf(u)|\, |\mu_v| &\leq \sup_{v\in V}\Big(|\mu_v|\sum_{u\in \Chi(v)}\Big|\frac{\lambda_u}{\mu_u}\Big|\Big)\sup_{u\in \Chi(v)}|f(u)\mu_u|\\
&\leq M \|f\|_{\ell^{\infty}(V,\mu)}.
\end{align*}
Thus we get that $B_{\lambda}$ is well-defined on $\ell^\infty(V,\mu)$ and 
\[
\|B_{\lambda}\|_{\ell^\infty(V,\mu)\to \ell^\infty(V,\mu)} \leq M,
\]
which implies (i). Since, in addition, 
\[
\|B_{\lambda}\|_{c_0(V,\mu)\to c_0(V,\mu)}\leq \|B_{\lambda}\|_{\ell^\infty(V,\mu)\to \ell^\infty(V,\mu)}
\]
we have together with \eqref{eq:4} that the norm of $B_\lambda$ on both spaces is $M$.
\end{proof}

As far as dynamical properties are concerned we will consider in this paper weighted shifts on unweighted spaces (as do Jab{\l}o\'nski et al.\ \cite{JJS12}) and unweighted shifts on weighted spaces (as does Martínez-Avenda\~{n}o \cite{Mar17}). One could perhaps argue that it is more natural to put weights on the operator, but the analysis is by far easier when the weights are put on the space because of the simplicity of the orbits of the unweighted shifts $B$ and $S$. Luckily, the two kinds of operators are canonically conjugate to one another and thus have the same dynamics. This is a well-known fact in the classical situation, see \cite{Shi74} and \cite[Section 4.1]{GrPe11}, and it is mentioned in \cite{Mar17} for arbitrary trees.

The conjugacies are in fact given by the following commutative diagrams. Let $\mu=(\mu_v)_{v\in V}$ and $\lambda=(\lambda_v)_{v\in V}$ be weights, and let $X$ denote one of the spaces $\ell^p(V)$, $1\leq p \leq\infty$, or $c_0(V)$. By $X_\mu$ we denote the space $\ell^p(V, \mu)$, $1\leq p \leq\infty$, or $c_0(V,\mu)$, respectively. Let $\phi_\mu(f)= (\mu_v f(v))_{v\in V}$. Then $\phi_\mu:X_\mu \to X$ is an isometric isomorphism.

Now, it is readily verified that
\[
\begin{CD}
X_\mu    @>S>>    X_\mu\\
@V{\phi_\mu}VV @VV{\phi_\mu}V \\
X    @>S_\lambda>> X
\end{CD}
\]
is a commutative diagram if the weights $\mu$ and $\lambda$ are linked by
\[
\lambda_v = \frac{\mu_v}{\mu_{\prt(v)}}, \, v\neq \rt;
\]
note that, in this diagram, $S$ maps $X_\mu$ into $X_\mu$ if and only if $S_\lambda$ maps $X$ into $X$.

In order to write $\mu$ in terms of $\lambda$ we fix a vertex $v_0\in V$; it is convenient to take $v_0=\rt$ if the tree is rooted. Then 
\[
\mu_v = \frac{\prod_{k=0}^{n-1}\lambda_{\prt^k(v)}}{\prod_{k=0}^{m-1}\lambda_{\prt^k(v_0)}}\mu_{v_0}, \, v\in V;
\]
here $m,n\geq 0$ are minimal numbers such that
\begin{equation}\label{eq-nm}
\prt^n(v) = \prt^m(v_0),
\end{equation}
and, as usual, empty products are 1. There is no harm to consider larger $m$ and $n$, but then there are extra terms in the quotient above that will cancel.

In the same way we have the commutative diagram
\begin{equation}\label{eq-conjB}
\begin{CD}
X_\mu    @>B>>    X_\mu\\
@V{\phi_\mu}VV @VV{\phi_\mu}V \\
X    @>B_\lambda>> X
\end{CD}
\end{equation}
if the weights $\mu$ and $\lambda$ are linked by
\[
\lambda_v = \frac{\mu_{\prt(v)}}{\mu_v}, \, v\neq \rt
\]
and
\begin{equation}\label{eq-conjBmu}
\mu_v = \frac{\prod_{k=0}^{m-1}\lambda_{\prt^k(v_0)}}{\prod_{k=0}^{n-1}\lambda_{\prt^k(v)}}\mu_{v_0}, \, v\in V
\end{equation}
with $m,n\geq 0$ minimal so that \eqref{eq-nm} holds.

Of course, more generally, one could also consider arbitrary weighted shifts on arbitrary weighted spaces; but they are likewise conjugate to an unweighted shift on a weighted space and therefore add nothing new. 

The conjugacies obtained above contain an interesting special case that is essentially known. Let $B_\lambda:X\to X$ and $B:X_\mu\to X_\mu$ be conjugate via \eqref{eq-conjBmu}, and let $c$ be a scalar with $|c|=1$. Then, again by \eqref{eq-conjBmu}, $cB_\lambda=B_{c\lambda}:X\to X$ is conjugate to $B:X_{\widetilde{\mu}}\to X_{\widetilde{\mu}}$, where the weight $\widetilde{\mu}$ satisfies $|\widetilde{\mu}_v/\widetilde{\mu}_{v_0}| = |\mu_v/\mu_{v_0}|$ for all $v\in V$. But this implies that $X_{\widetilde{\mu}}=X_\mu$. In other words, $c B_\lambda$ and $B_\lambda$ are similar operators on $X$ whenever $|c|=1$. As a consequence, in the complex scalar case, the spectrum and the point spectrum of $B_\lambda$ are rotation invariant. The same arguments work for any weighted shift on any weighted space. Moreover, one sees that the conjugacy is given by an isometric isomorphism. In the case of Hilbert spaces, such operators are called \textit{circular operators}. For simplicity, let us use this terminology also in the Banach space setting.

\begin{proposition}
Let $\lambda$ and $\mu$ be weights. Then any weighted shift operator $B_\lambda$ and $S_\lambda$ is a circular operator on any of the spaces $\ell^p(V,\mu)$, $1\leq p\leq \infty$, and $c_0(V,\mu)$.
\end{proposition}

In the Hilbert space case, this result is due to Shields \cite[Section 2]{Shi74} for classical weighted shifts, and to Jab{\l}o\'nski et al. \cite[Section 3.3]{JJS12} for shifts on directed trees.

\subsection{Rolewicz operators and symmetric shift operators} 
In this paper we will consider two special classes of backward shift operators, the multiples of the backward shift $B$, also known as Rolewicz operators, and symmetric shift operators. 

\begin{example}\label{ex-rol}
Let $\lambda\in\KK$, $\lambda\neq 0$. Then the \textit{Rolewicz operator} $\lambda B=B_{(\lambda)_{v\in V}}$ is bounded on each space $\ell^1(V)$, while it is bounded on any of the spaces $\ell^p(V)$, $1<p\leq \infty$, or $c_0(V)$ if and only if $\sup_{v\in V} |\Chi(v)|<\infty$. These operators are named after S. Rolewicz who was the first to consider the dynamical properties of the operator $\lambda B$ for $V=\NN$, see \cite{Rol}. 
\end{example}

A weight $\lambda=(\lambda_v)_{v\in V}$ will be called \textit{symmetric} if 
\[
\lambda_u=\lambda_v
\]
whenever $u\sim v$. That is, a symmetric weight is a function of the generation. Let us fix $v_0\in V$, with $v_0=\rt$ in the rooted case, and let us enumerate the generations with respect to $v_0$, see Subsection \ref{subs-dirtr}. For $v\in \gen_n$ we will use the notation
\[
\lambda_n=\lambda_v.
\]
Notice that this makes sense due to the symmetry of the weight. Thus, from now on, we will be indexing a symmetric weight as $\lambda=(\lambda_n)_{n}$, where $n$ ranges over $\NN_0$ in the rooted case and over $\ZZ$ in the unrooted case. An interesting class of weighted backward shifts, generalizing the usual unilateral and bilateral ones, occurs when we consider a symmetric weight on a symmetric directed tree. 

\begin{example}\label{symmetric:boundedness}
Let $(V,E)$ be a symmetric directed tree and $\lambda=(\lambda_n)_{n}$ a symmetric weight on it. Then $B_{\lambda}$ is bounded on $\ell^1(V)$ if and only if 
\[
\sup_{n}|\lambda_n|<\infty,
\]
while it is bounded on any of the spaces $\ell^p(V)$, $1<p\leq \infty$, or $c_0(V)$, if and only if
\[
\sup_{n}\gamma_n|\lambda_{n+1}|^{p^*}<\infty,
\]
where $p^*=1$ when the underlying space is $\ell^{\infty}(V)$ or $c_0(V)$. We note that for $\ell^1(V)$ the symmetry of the directed tree is not needed. 
\end{example}

\section{Dynamics of weighted forward shifts}

Hypercyclicity of forward shifts has already been treated conclusively by Mar\-tínez-Aven\-da\~{n}o \cite[Section 3]{Mar17} (as said before, he considered the unweighted shift on weighted spaces). We give here the proofs for the sake of completeness.

If the directed tree $(V,E)$ has a root then $S_{\lambda}$ is not hypercyclic on $\ell^p(V)$, $1\leq p<\infty$, or on $c_0(V)$, see \cite[Proposition 3.3]{Mar17}. This is due to the fact that, for any $n\geq 1$,
\[
(S_{\lambda}^nf)(\rt)=0.
\]

\begin{proposition}[{\cite[Proposition 3.4]{Mar17}}]
If $V$ is unrooted and has a vertex of outdegree at least $2$ then $S_{\lambda}$ is not hypercyclic on $\ell^p(V)$, $1\leq p<\infty$, or on $c_0(V)$.
\end{proposition}

\begin{proof}
We consider the corresponding unweighted shift $S$ on a weighted space. Let $v_1, v_2$ be two different children of a vertex $v$, and let $f$ be an arbitrary element of the underlying space. For each $n\geq 0$ we then have that
\[
(S^n f)(v_k)=f(\prt^n(v_k)),\, k=1,2.
\]
Since $\prt(v_1)=\prt(v_2)$, we see that 
\[
(S^nf)(v_1)=(S^nf)(v_2)
\]
for all $n\geq 1$, which implies that $f$ cannot be hypercyclic for $S$.
\end{proof}

From the above it follows that for $S_{\lambda}$ to be hypercyclic, $V$ must be $\ZZ$. The dynamics of weighted shifts on $\ZZ$ is rather well understood, see for instance \cite{BM09} and \cite{GrPe11}.

The remainder of this paper is devoted to the study of the dynamical properties of weighted backward shifts on directed trees.

\section{Dynamics of weighted backward shifts - the rooted case}\label{s-root}

Before we continue, let us state the obvious.

\begin{remark}\label{r-leaf}
If a directed tree $(V,E)$ has a leaf, then no weighted backward shift can be hypercyclic on any of the spaces considered here. This follows from the fact that, if $v$ is a leaf, then $(B_\lambda^nf)(v)=0$ for any element $f$ and any $n\geq 1$.
\end{remark}

\textit{Thus, throughout Sections \ref{s-root} and \ref{s-unrooted} we will assume that the directed tree has no leaves.}
\vspace{\baselineskip}

The dynamics of weighted backward shifts on $\NN$ is considerably easier to determine than that for shifts on $\ZZ$. It turns out that the same is true more generally for rooted trees as opposed to unrooted ones. 

When trying to characterize hypercyclicity of backward shifts, we were confronted with the need to control the lower bound of a weighted norm when the unweighted $\ell^1$-norm is fixed. This is exactly what the reverse H\"older inequality does.

\begin{lemma}\label{l-revhol}
Let $J$ be a finite or countable set. Let $\mu=(\mu_j)_j \in (\KK\setminus \{0\})^J$. Then
     \begin{align*}   
		\inf_{\|x\|_1=1} \sum_{j\in J} |x_j \mu_j| &=  \inf_{j\in J} |\mu_j|,\\
		\inf_{\|x\|_1=1} \Big(\sum_{j\in J} |x_j \mu_j|^p\Big)^{1/p} &=  \Big(\sum_{j\in J} \frac{1}{|\mu_j|^{p^*}}\Big)^{-1/p^*},\, 1<p<\infty,\\
  	\inf_{\|x\|_1=1} \sup_{j\in J} |x_j \mu_j| &=  \Big(\sum_{j\in J} \frac{1}{|\mu_j|}\Big)^{-1},
     \end{align*}
where $x\in \KK^J$, $\|x\|_1=\sum_{j\in J}|x_j|$ and $\infty^{-1}=0$.

The same holds when the sequences $x$ are required, in addition, to be of finite support.
\end{lemma}

\begin{proof}
The first claim is easily seen. 

Next, let $1<p<\infty$, and suppose that $J$ is a finite set. Take $q=\frac{1}{1-p}$, so that $\frac{1}{1/p}+\frac{1}{q}=1$. Then the reverse H\"older inequality, see \cite[Theorem 13]{HLP34}, tells us that
\[
\sum_{j\in J} |x_j \mu_j|^p \geq  \Big(\sum_{j\in J} |x_j|^{p/p}\Big)^p \Big(\sum_{j\in J} |\mu_j|^{pq}\Big)^{1/q}, 
\]
and hence
\[
\Big(\sum_{j\in J} |x_j \mu_j|^p\Big)^{1/p} \geq  \sum_{j\in J} |x_j|\Big(\sum_{j\in J} |\mu_j|^{-p^*}\Big)^{-1/p^*}, 
\]
where equality holds for $x=(|\mu_j|^{-p^*})_j$. Dividing by $\|x\|_1$ we get the second claim for finite sets $J$. The general case is easily deduced from this.

For the third claim: Since
\[
\sum_{j\in J} |x_j|=\sum_{j\in J} \Big(|x_j\mu_j| \frac{1}{|\mu_j|}\Big) \leq (\sup_{j\in J} |x_j \mu_j|) \sum_{j\in J} \frac{1}{|\mu_j|},
\]
it follows that the infimum is at least $\big(\sum_{j\in J} \frac{1}{|\mu_j|}\big)^{-1}$. If $J$ is finite, then the infimum is attained for
$x\in \KK^J$ given by
\[
x_j = \frac{1}{|\mu_j|} \Big(\sum_{k\in J} \frac{1}{|\mu_k|}\Big)^{-1},\quad j\in J.
\]
The general case is easily deduced from this.
\end{proof}

One may also arrive at the second claim for finite sets $J$ by the method of Lagrange multipliers: one needs to minimize $\sum_{j\in J} x_j^p |\mu_j|^p$ subject to $x_j >0$ for all $j\in J$ and $\sum_{j\in J} x_j =1$.

We can now characterize hypercyclic backward shifts on directed rooted trees, and we first put the weights on the space. As is well known, for the shifts on $\NN$, hypercyclicity and weak mixing coincide. The following theorem shows that this property is inherited by shifts on trees.

\begin{theorem}\label{t-charHCroot}
Let $(V,E)$ be a rooted directed tree and $\mu$ a weight. Let $X=\ell^p(V,\mu)$, $1\leq p<\infty$, or $X=c_0(V,\mu)$, and suppose that the backward shift $B$ is a bounded operator on $X$. 

{\rm (a)} The following assertions are equivalent:
\begin{enumerate}[{\rm (i)}]
    \item $B$ is hypercyclic;
    \item $B$ is weakly mixing;
    \item there is an increasing sequence $(n_k)_k$ of positive integers such that, for each $v\in V$, we have as $k\to\infty$,
		\begin{align*}
		\inf_{u\in \Chi^{n_k}(v)} |\mu_u| \to 0, &\text{ if $X=\ell^1(V,\mu)$;}\\
    \sum_{u\in \Chi^{n_k}(v)} \frac{1}{|\mu_u|^{p^*}} \to \infty, &\text{ if $X=\ell^p(V,\mu)$, $1<p<\infty$;}\\
    \sum_{u\in \Chi^{n_k}(v)} \frac{1}{|\mu_u|} \to \infty, &\text{ if $X=c_0(V,\mu)$.}
    \end{align*}
\end{enumerate}

{\rm (b)} $B$ is mixing if and only if condition {\rm (iii)} holds for the full sequence $(n_k)_k=(n)_n$.
\end{theorem}

\begin{proof}
(a) It suffices to show that (iii) $\Rightarrow$ (ii) and (i) $\Rightarrow$ (iii). The strategy of the proof is the classical one, the new element being Lemma \ref{l-revhol} (for both implications). 

We do the proof for $X=\ell^p(V,\mu)$, $1<p<\infty$; the other two cases are similar.

(iii) $\Rightarrow$ (ii). We apply the Hypercyclicity Criterion. Let $X_0=Y_0=\vect\{e_v: v\in V\}$, which is dense in $\ell^p(V,\mu)$. 

Let $v\in V$. By condition (iii) and Lemma \ref{l-revhol}, applied to $J=\Chi^{n_k}(v)$, there are $g_{v,n_k}\in \KK^V$, $k\geq 1$,
of support in $\Chi^{n_k}(v)$ such that
\begin{equation}\label{eq-one}
\sum_{u\in \Chi^{n_k}(v)}|g_{v,n_k}(u)|=1
\end{equation}
and
\[
\sum_{u\in \Chi^{n_k}(v)}|g_{v,n_k}(u)\mu_u|^p\to 0
\]
as $k\to\infty$. Then $g_{v,n_k}\in \ell^p(V,\mu)$ and $g_{v,n_k}\to 0$ in $\ell^p(V,\mu)$ as $k\to\infty$. We can suppose that $g_{v,n_k}$ is non-negative. (For $X=c_0(V,\mu)$ we choose $g_{v,n_k}$ to be in addition of finite support so that $g_{v,n_k}\in c_0(V,\mu)$; see the final assertion in Lemma \ref{l-revhol}.)

Now  define maps $R_{n_k}:Y_0\rightarrow \ell^p(V,\mu)$, $k\geq 1$, by setting
\[
R_{n_k}e_v=g_{v,n_k}, v\in V
\]
and extending linearly to $Y_0$. Then, for all $g\in Y_0$, 
\[
R_{n_k}g \to 0
\]
as $k\to\infty$. In addition, since $g_{v,n_k}$ is non-negative and has support in $\Chi^{n_k}(v)$, \eqref{eq-one} implies that
\[
B^{n_k}R_{n_k}e_v=B^{n_k} g_{v,n_k} =e_v, v\in V,
\]
so that $B^{n_k}R_{n_k}$ is the identity on $Y_0$; note that the second equation is because $B$ is the unweighted shift. 

Finally, since $V$ is rooted, for any $f\in X_0$ there is some $k$ large enough so that
\[
B^{n_k}f=0.
\]
Thus $B$ satisfies the Hypercyclicity Criterion and is therefore weakly mixing.

(i) $\Rightarrow$ (iii). Let $F\subset V$ be finite, $N\geq 1$ and $\eps >0$. Choose $\delta>0$ such that, for any $g=(g(v))_{v\in V}\in \ell^p(V, \mu)$, $\|g-\sum_{v\in F}e_v\|<\delta$ implies that $|g(v)|\geq \tfrac{1}{2}$ for all $v\in F$.

Now, by topological transitivity of $B$ there is some $f\in \ell^p(V, \mu)$ and some $n\geq N$ such that
\[
\|f\|< {\eps}\quad\text{and}\text \quad \Big\|B^nf-\sum_{v\in F}e_v\Big\|<\delta,
\]
which implies that, for all $v\in F$,
\[
c_v:=\sum_{u\in \Chi^{n}(v)}|f(u)|\geq \Big|\sum_{u\in \Chi^{n}(v)}f(u)\Big|=|(B^nf)(v)|\geq \frac{1}{2}.
\]
Choosing, in Lemma \ref{l-revhol}, $\Chi^{n}(v)$ for $J$ and the restriction of $\frac{1}{c_v}|f|$ to $\Chi^{n}(v)$ as $x$ we obtain that, for any $v\in F$,
\[
\Big(\sum_{u\in \Chi^{n}(v)} \frac{1}{|\mu_u|^{p^*}}\Big)^{-1/p^*} \leq \frac{1}{c_v}\Big(\sum_{u\in \Chi^{n}(v)} |f(u) \mu_u|^p\Big)^{1/p}\leq \frac{1}{c_v}\|f\|<2\eps.
\]

Applying this argument now to an increasing sequence $(F_k)_k$ of finite subsets of $V$ with $\bigcup_{k\geq 1}F_k =V$, to an increasing sequence $(N_k)_k$ of positive integers, and to a positive sequence $(\eps_k)_k$ tending to 0, condition (iii) follows.

(b) The proof is identical to that of (a). For one implication one has to note that if the Hypercyclicity Criterion holds for the full sequence then the operator is mixing. For the other implication one replaces topological transitivity by the definition of the mixing property.
\end{proof}

For what it's worth let us mention that the characterizing conditions in (iii) can be written in a unified way. If we define the conjugate exponent as $p^*=1$ for $X=c_0(V,\mu)$, then the three conditions are in fact
\[
\Big\| \Big(\frac{1}{\mu_u}\Big)_{u\in\Chi^{n_k}(v)}\Big\|_{p^*}\to \infty,
\]
where $\|\cdot\|_{p^*}$ denotes the usual $p^*$-norm.

We can now pass to weighted shifts on unweighted spaces via the conjugacy \eqref{eq-conjB} with \eqref{eq-conjBmu}. In order to simplify notation, let us write for a weight $\lambda=(\lambda_v)_{v\in V}$
\[
\lambda(v\to u) := \prod_{k=0}^{n-1} \lambda_{\prt^k(u)} = \lambda_{\prt^{n-1}(u)}\cdots\lambda_{\prt^2(u)}\lambda_{\prt(u)}\lambda_{u}
\]
whenever $u\in \Chi^n(v)$: it is simply the product of the $\lambda_w$ for all the vertices $w$ along the (unique) branch leading from $v$ to $u$, see Figure \ref{f-branch}. Note that we have the identity
\begin{equation}\label{eq-lambda}
\lambda(v\to u) = \lambda(v\to w) \lambda(w\to u)
\end{equation}
for any vertex $w$ on the branch from $v$ to $u$.

\begin{figure}[h!]%
\begin{tikzpicture}[scale=1]
\draw[->,>=latex] (0,0) - - (1,.8);
\draw[->,>=latex] (1,.8) - - (2,1.3);
\draw[->,>=latex] (2,1.3) - - (3,1.7);
\draw[->,>=latex] (3,1.7) - - (4,2);
\draw[->,>=latex] (4,2) - - (5,2.2);
\draw[->,>=latex] (5,2.2) - - (6,2.3);
\draw[fill] (0,0) circle (1pt);
\draw (.8,0.1) node[align=center, below]{{$v=\prt^n(u)$}};
\draw (-.3,.5) node{{\scriptsize $\lambda_{\prt^{n-1}(u)}$}};
\draw[fill] (1,.8) circle (1pt);
\draw (1.6,0.9) node[align=center, below]{{\scriptsize $\prt^{n-1}(u)$}};
\draw (.8,1.25) node{{\scriptsize $\lambda_{\prt^{n-2}(u)}$}};
\draw[fill] (2,1.3) circle (1pt);
\draw[fill] (3,1.7) circle (1pt); 
\draw (3.35,1.8) node[align=center, below]{{\scriptsize $\prt^3(u)$}};
\draw (3.0,2.05) node{{\scriptsize $\lambda_{\prt^{2}(u)}$}};
\draw[fill] (4,2) circle (1pt); 
\draw (4.3,2.1) node[align=center, below]{{\scriptsize $\prt^2(u)$}};
\draw (4.1,2.3) node{{\scriptsize $\lambda_{\prt(u)}$}};
\draw[fill] (5,2.2) circle (1pt); 
\draw (5.25,2.3) node[align=center, below]{{\scriptsize $\prt(u)$}};
\draw (5.4,2.45) node{{\scriptsize $\lambda_u$}};
\draw[fill] (6,2.3) circle (1pt);
\draw (6.25,2.3) node{{$u$}};
\end{tikzpicture}
\caption{A branch of a directed tree with associated weights}%
\label{f-branch}%
\end{figure}
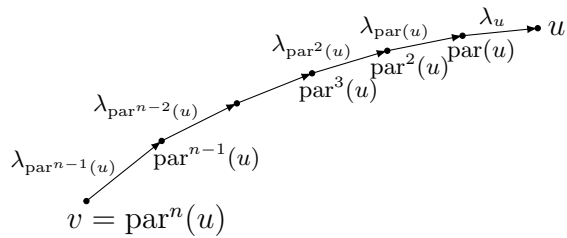

\begin{theorem}\label{t-charHCrootalt}
Let $(V,E)$ be a rooted directed tree and $\lambda$ a weight. Let $X=\ell^p(V)$, $1\leq p<\infty$, or $X=c_0(V)$, and suppose that the weighted backward shift $B_\lambda$ is a bounded operator on $X$. 

{\rm (a)} The following assertions are equivalent:
\begin{enumerate}[{\rm (i)}]
    \item $B_{\lambda}$ is hypercyclic;
    \item $B_{\lambda}$ is weakly mixing;
    \item there is an increasing sequence $(n_k)_k$ of positive integers such that, for each $v\in V$, we have as $k\to\infty$,
		\begin{align*}
		\sup_{u\in \Chi^{n_k}(v)} |\lambda(v\to u)| \to\infty, &\text{ if $X=\ell^1(V)$;}\\
    \sum_{u\in \Chi^{n_k}(v)} |\lambda(v\to u)|^{p^*} \to \infty, &\text{ if $X=\ell^p(V)$, $1<p<\infty$;}\\
    \sum_{u\in \Chi^{n_k}(v)} |\lambda(v\to u)| \to \infty, &\text{ if $X=c_0(V)$.}
    \end{align*}
\end{enumerate}

{\rm (b)} $B_\lambda$ is mixing if and only if condition {\rm (iii)} holds for the full sequence $(n_k)_k=(n)_n$.
\end{theorem}

As our first application we consider the Rolewicz operators, see Example \ref{ex-rol}. Here the geometry of the tree plays a crucial role, unless the underlying space is $\ell^1(V)$. 

\begin{corollary}\label{c-rol}
Let $(V,E)$ be a rooted directed tree. Let $\lambda\in\KK$ be a non-zero number and $\lambda B$ the corresponding Rolewicz operator.

{\rm (a)} The operator $\lambda B$ is a bounded hypercyclic operator on $\ell^1(V)$ if and only if $|\lambda|>1$. In that case it is mixing.

{\rm (b)} Let $X=\ell^p(V)$, $1< p<\infty$, or $X=c_0(V)$, and let $\lambda B$ be bounded on $X$. Then $\lambda B$ is hypercyclic on $X$ if and only if there is an increasing sequence $(n_k)_k$ of positive integers such that, for each $v\in V$,
\[
|\Chi^{n_k}(v)|\,|\lambda|^{n_kp^*} \to\infty
\]
as $k\to\infty$. And $\lambda B$ is mixing on $X$ if and only if, for each $v\in V$,
\[
|\Chi^{n}(v)|\,|\lambda|^{np^*}\to\infty
\]
as $n\to\infty$. Here, $p^*=1$ for $X=c_0(V)$.
\end{corollary}

As a particular case we note that if the directed tree $(V,E)$ is the rooted tree in which each vertex has exactly $N\geq 1$ children, then the Rolewicz operator $\lambda B$ is hypercyclic on $X$ (and then mixing) if and only if 
\begin{align*}
|\lambda|>1 & \text{ if $X=\ell^1(V)$};\\
|\lambda|>N^{-1/p^*} & \text{ if $X=\ell^p(V)$, $1<p<\infty$};\\
|\lambda|>N^{-1} & \text{ if $X=c_0(V)$}.
\end{align*}

As another special case we have the following.

\begin{example}
Interestingly, there are rooted directed trees $(V,E)$ that support a hypercyclic non-mixing Rolewicz operator on $\ell^p(V)$, $1<p<\infty$, or $c_0(V)$. One need only choose a scalar $\lambda$ with $0<|\lambda|<1$, fix an integer $N>|\lambda|^{-p^*}$, and then construct a symmetric tree in which sufficiently long stretches of generations with exactly one child alternate with sufficiently long stretches of generations with exactly $N$ children. 
\end{example}

In the classical case of the rooted tree $V=\mathbb{N}$ it suffices to demand the conditions in (iii) of Theorem \ref{t-charHCrootalt} only for the root, see \cite[Theorem 4.8]{GrPe11}. This is, of course, not the case for general trees. Just consider the directed tree in which the root has exactly two children and any other vertex has exactly one child; then define $\lambda_v=1$ along one branch and $\lambda_v=2$ along the other. However, when we impose symmetry on both the directed tree and the weight, then the conditions characterizing hypercyclicity and the mixing property simplify just as in classical case. 

\begin{corollary}
Let $(V,E)$ be a symmetric, rooted directed tree and $\lambda=(\lambda_n)_{n\in \NN_0}$ a symmetric weight. Let $X=\ell^p(V)$, $1\leq p<\infty$, or $X=c_0(V)$, and suppose that $B_{\lambda}$ is a bounded operator on $X$.

{\rm (a)} The operator $B_{\lambda}$ is hypercyclic (resp. mixing) on $\ell^1(V)$ if and only if
\[
\sup_{n\geq 1}|\lambda_1 \cdots \lambda_n|=\infty \quad (resp. \lim_{n\rightarrow \infty}|\lambda_1 \cdots \lambda_n|=\infty).
\]

{\rm (b)} The operator $B_{\lambda}$ is hypercyclic (resp. mixing) on $\ell^p(V)$, $1<p<\infty$, or on $c_0(V)$ if and only if
\[
\sup_{n\geq 1}\gamma_0 \cdots \gamma_{n-1}|\lambda_1 \cdots \lambda_n|^{p^*}=\infty \quad (resp. \lim_{n\rightarrow \infty}\gamma_0 \cdots \gamma_{n-1}|\lambda_1 \cdots \lambda_n|^{p^*}=\infty),
\]
where $p^*=1$ for $X=c_0(V)$.
\end{corollary}

\begin{proof}
We will do the proof for hypercyclicity on $\ell^p(V)$, $1<p<\infty$, since the remaining cases are similar and even simpler. Let us observe that, due to the symmetry of $(V,E)$, if $v\in \gen_m$, $m\geq 0$, then
\[
|\Chi^n(v)|=\gamma_m \gamma_{m+1}\cdots \gamma_{m+n-1}, \quad n\geq 1.
\]

Now, one implication of the claim is obvious since if $B_{\lambda}$ is hypercyclic, then by Theorem \ref{t-charHCrootalt} there is an increasing sequence $(n_k)_k$ such that
\[
|\Chi^{n_k}(\rt)|\,|\lambda_1\cdots \lambda_{n_k}|^{p^*}=\gamma_0 \cdots \gamma_{n_k-1}|\lambda_1 \cdots \lambda_{n_k}|^{p^*}\rightarrow \infty.
\]

For the reverse implication, let $(n_k)_k$ be an increasing sequence such that 
\[
\gamma_0 \cdots \gamma_{n_k-1}|\lambda_1 \cdots \lambda_{n_k}|^{p^*}\rightarrow \infty.
\]
We fix $j\in \NN$. We then have for $k$ sufficiently large that
\[
\gamma_0 \cdots \gamma_{n_k-j-1}|\lambda_1 \cdots \lambda_{n_k-j}|^{p^*}=\frac{\gamma_0 \cdots \gamma_{n_k-1}|\lambda_1 \cdots \lambda_{n_k}|^{p^*}}{\gamma_{n_k-j} \cdots \gamma_{n_k-1}|\lambda_{n_k-j+1} \cdots \lambda_{n_k}|^{p^*}}.
\]
By Proposition \ref{p-Bbounded}, the boundedness of $B_{\lambda}$ yields for these $k$ that
\[
\gamma_{n_k-j} \cdots \gamma_{n_k-1}|\lambda_{n_k-j+1} \cdots \lambda_{n_k}|^{p^*}\leq \|B_{\lambda}\|^{jp^*}.
\]
We therefore deduce that, for each $j\in \NN$,
\[
\gamma_0 \cdots \gamma_{n_k-j-1}|\lambda_1 \cdots \lambda_{n_k-j}|^{p^*}\rightarrow \infty.
\]
By an obvious variant of \cite[Lemma 4.2]{GrPe11} we obtain an increasing sequence $(m_k)_k$ such that, for each $j\in \NN$,
\[
\gamma_0 \cdots \gamma_{m_k+j-1}|\lambda_1 \cdots \lambda_{m_k+j}|^{p^*}\rightarrow \infty,
\]
which implies that, for each $v\in \gen_j$,
\[
|\Chi^{m_k}(v)|\,|\lambda_{j+1} \cdots \lambda_{m_k+j}|^{p^*}\rightarrow \infty.
\]
By Theorem \ref{t-charHCrootalt}, this shows that $B_{\lambda}$ is hypercyclic.
\end{proof}

Incidentally, as the proof shows, the symmetry of the directed tree is not needed when $X=\ell^1(V)$.

\begin{example}
If $(V,E)$ is a rooted directed tree with finitely many children for each vertex and $q\geq 1$ a real constant, the \textit{Dirichlet shift} according to \cite{CPT17} is the weighted forward shift $S_{\lambda,q}$ on $\ell^2(V)$, where
$$
\lambda_{u,q}=\frac{1}{\sqrt{|\Chi(v)|}}\sqrt{\frac{n_v+q}{n_v+1}}
$$
for $v\in \gen_{n_v}$ and $u\in \Chi(v)$. Then $S_{\lambda,q}$ is bounded with $\|S_{\lambda,q}\|=\sqrt{q}$. Being a weighted forward shift on a rooted tree, the Dirichlet shift is not hypercyclic. The adjoint of $S_{\lambda,q}$, which in the unilateral case with $q=2$ is called the \textit{Bergman backward shift} \cite{BoSh00}, is $B_{\lambda,q}$. When $q=1$, $B_{\lambda,q}$ is not hypercyclic since $\|B_{\lambda,q}\|=1$. If $q>1$, for $n\in \mathbb{N}$ and $v\in V$, we have that
$$
\sum_{u\in \Chi^n(v)}|\lambda_q(v\rightarrow u)|^2=\frac{\prod_{k=0}^{n-1}(n_v+q+k)}{\prod_{k=1}^n(n_v+k)}\sum_{u\in \Chi^n(v)}\frac{1}{\prod_{k=1}^n|\Chi(\prt^k(u))|}.
$$
Now, 
$$
\sum_{u\in \Chi^n(v)}\frac{1}{\prod_{k=1}^n|\Chi(\prt^k(u))|}=1
$$
and
$$
\frac{\prod_{k=0}^{n-1}(n_v+q+k)}{\prod_{k=1}^n(n_v+k)}=\frac{n_v!}{\Gamma(n_v+q)}\frac{\Gamma(n_v+q+n)}{(n_v+n)!}\sim \frac{n_v!}{\Gamma(n_v+q)}n^{q-1}\rightarrow \infty.
$$
By Theorem \ref{t-charHCrootalt}, we deduce that $B_{\lambda,q}$ is mixing in this case.
\end{example}

Returning to the general situation, we see that the behaviour of the numbers
\begin{equation}\label{eq-Lambda}
\Lambda_{v,n} = \sum_{u\in \Chi^{n}(v)} |\lambda(v\to u)|,\; v\in V, n\geq 1
\end{equation}
determines hypercyclicity (and the mixing property) of $B_\lambda$ on $c_0(V)$. But they also play a certain role for the other spaces.
Part (a) of the following result is due to Martínez-Avenda\~{n}o \cite[Theorem 5.4]{Mar17}. An analogous result holds for the mixing property.

\begin{corollary}\label{c-charHCrootalt}
Let $(V,E)$ be a rooted directed tree and $\lambda$ a weight. Let  $B_\lambda$ be a bounded operator on $\ell^p(V)$, $1\leq p<\infty$. 

{\rm (a)} If $B_\lambda$ is hypercyclic, then there is an increasing sequence $(n_k)_k$ of positive integers such that, for each $v\in V$, 
\[
\Lambda_{v,n_k} \to \infty
\]
as $k\to\infty$.

{\rm (b)} Suppose that every vertex has at most finitely many children. If there is an increasing sequence $(n_k)_k$ of positive integers such that, for each $v\in V$,
\[
\frac{1}{|\Chi^{n_k}(v)|^{1/p}}\Lambda_{v,n_k}\to \infty
\]
as $k\to\infty$, then $B_\lambda$ is hypercyclic.
\end{corollary}

\begin{proof} (a) This follows from Theorem \ref{t-charHCrootalt} and the monotonicity of $p$-norms ($1\leq p\leq\infty$) on finite-dimensional spaces $\KK^J$.

(b) This follows from Theorem \ref{t-charHCrootalt} and H\"older's inequality, which implies that
\[
\Lambda_{v,n}\leq \Big(\sum_{u\in \Chi^{n}(v)} |\lambda(v\to u)|^{p^*}\Big)^{1/p^*} |\Chi^{n}(v)|^{1/p},
\]
with the usual interpretation when $p=1$.
\end{proof}

It is easy to see that neither condition in the corollary is a characterization.

\begin{example} 
Let $(V,E)$ be the rooted binary tree, that is, the directed tree with a root for which every vertex has exactly two children. Let $1\leq p<\infty$.

(a) There is a non-hypercyclic weighted backward shift $B_\lambda$ on $\ell^p(V)$ such that, for any $v\in V$,
\[
\Lambda_{v,n}\to \infty\; \text{ as $n\to\infty$}.
\]
An example is the Rolewicz operator $\lambda B$ with $\frac{1}{2}<|\lambda|\leq 1$ if $p=1$ and $\frac{1}{2}<|\lambda|\leq 2^{-1/p^*}$ if $p>1$, see Corollary \ref{c-rol}.

(b) There is a hypercyclic (even mixing) weighted backward shift $B_\lambda$ on $\ell^p(V)$, $1\leq p<\infty$, such that, for any $v\in V$,
\[
\frac{1}{|\Chi^{n}(v)|^{1/p}}\Lambda_{v,n}\to 0\; \text{ as $n\to\infty$}.
\]
Let us define the weight $\lambda$ such that, for each $v\in V$, one of its two children has weight $a>0$, the other one has weight $b>0$. Let $v\in V$. Since there are $\binom{n}{k}$ branches of length $n$ that start form $v$ and carry exactly $k$ times the weight $a$, we have that
\[
\frac{1}{|\Chi^{n}(v)|^{1/p}}\Lambda_{v,n}= \frac{1}{2^{n/p}} \sum_{k=0}^n \binom{n}{k}a^kb^{n-k} = \frac{(a+b)^n}{2^{n/p}},
\]
while, similarly,
\[
\sum_{u\in \Chi^{n}(v)} |\lambda(v\to u)|^{p^*} = (a^{p^*}+b^{p^*})^n,\quad p>1,
\]
which has to be replaced by $\sup_{u\in \Chi^{n}(v)} |\lambda(v\to u)| = \max(a,b)^n$ if $p=1$. Now choose $a$ and $b$ so that the first expression tends to 0 and the second tends to $\infty$ as $n\to\infty$.
\end{example}

\begin{remark}\label{r-mart_suff}
Martínez-Avenda\~{n}o \cite[Theorem 5.1]{Mar17} has obtained another sufficient condition for hypercyclicity of backward shifts; his condition can  be deduced by another application of H\"older's inequality on Theorem \ref{t-charHCrootalt}. Indeed, suppose again that every vertex has at most finitely many children. If $B_\lambda$ is a bounded operator on $X=\ell^p(V)$, $1\leq p <\infty$, or $X=c_0(V)$, and if there is an increasing sequence $(n_k)_k$ of positive integers such that, for each $v\in V$,
\begin{align*}
    \frac{1}{|\Chi^{n_k}(v)|^p}\sum_{u\in \Chi^{n_k}(v)} \frac{1}{|\lambda(v\to u)|^{p}} \to 0, &\text{ if $X=\ell^p(V)$, $1\leq p<\infty$,}\\
    |\Chi^{n_k}(v)|\inf_{u\in \Chi^{n_k}(v)} |\lambda(v\to u)| \to \infty, &\text{ if $X=c_0(V)$}
\end{align*}
as $k\to\infty$, then $B_\lambda$ is hypercyclic on $X$. The analogous result holds for the mixing property.
\end{remark}

\begin{example}
There is a natural way to join two classical weighted backward shifts $T_1$ and $T_2$ on $\ell^p(\NN)$, $1\leq p<\infty$, or $c_0(\NN)$ into a single weighted backward shift on a rooted directed tree $V$. Indeed, take $V$ to be the tree where only the root has two children, while all other vertices have exactly one child. Then we can identify the two branches of $V$ with two copies of $\NN=\{1,2,3,\ldots\}$, with the two elements 1 being the children of the root, and we put on the branches the shifts $T_1$ and $T_2$, respectively. Then, by Theorem \ref{t-charHCrootalt}
and \cite[Theorem 2.8]{Salas}, the resulting weighted shift $B_\lambda$ on $V$ (with $B_\lambda f(\rt)$ defined arbitrarily) is easily seen to be hypercyclic if and only if the direct sum $T_1\oplus T_2$ given by $(T_1\oplus T_2)(x,y)=(T_1x,T_2y)$ is hypercyclic. 

It is interesting to note that there are hypercyclic weighted backward shifts $T_1$ and $T_2$ on $\ell^2(\NN)$ so that $T_1\oplus T_2$ is not hypercyclic, see \cite[unilateral version of Corollary 2.6]{Salas}, in which case the joined shift $B_\lambda$ is not hypercyclic. Of course, if one of the operators $T_1$ or $T_2$ is even mixing, then $T_1\oplus T_2$ is hypercyclic and then so is $B_\lambda$.
\end{example}

A natural question arises: given any rooted directed tree, does it support a hypercyclic or even mixing weighted backward shift? We will consider this question for arbitrary directed trees in Section \ref{s-exist}.

\section{Dynamics of weighted backward shifts - the unrooted case}\label{s-unrooted}

We now address the case of unrooted trees. The situation here is substantially more complicated. We first need the following variant of the Hypercyclicity Criterion, see Theorem \ref{t-hypcrit}. The usual condition that, for all $x\in X_0$,
\[
T^{n_k}x\to 0,\quad k\to\infty,
\]
is replaced by one of the form
\[
T^{n_k}x_k\to 0,\quad k\to\infty,
\]
where $x_k\to x$.

\begin{proposition}\label{p-criterion}
Let $X$ be a separable Banach space and $T$ a bounded operator on $X$. Suppose that there exist dense subsets $X_0, Y_0$ of $X$, an increasing sequence $(n_k)_k$ of positive integers, and maps $I_{n_k}:X_0\rightarrow X$ and $R_{n_k}:Y_0\rightarrow X$, $k\geq 1$, such that, for any $x\in X_0$ and $y\in Y_0$,
\begin{enumerate}[{\rm (i)}]
    \item $I_{n_k}x\rightarrow x$, 
    \item $T^{n_k}I_{n_k}x\rightarrow 0$,
    \item $R_{n_k}y\rightarrow 0$, and
    \item $T^{n_k}R_{n_k}y\rightarrow y$
\end{enumerate}
as $k\to\infty$. Then $T$ is weakly mixing.

If the conditions hold for the full sequence $(n_k)_k=(n)_n$ then $T$ is mixing.
\end{proposition}

\begin{proof}
Let $U_1, U_2, V_1, V_2$ be non-empty open subsets of $X$. Find $x_j\in U_j\cap X_0$ and $y_j\in V_j\cap Y_0$ and set $z_{j,k}:=I_{n_k}x_j+R_{n_k}y_j$ for $j=1,2$, $k\geq 1$. Then $z_{j,k}\rightarrow x_j$ and $T^{n_k}z_{j,k}\rightarrow y_{j}$ as $k\rightarrow \infty$, so that 
there is some $k\geq 1$ such that $T^{n_k}(U_j)\cap V_j\neq \varnothing$, $j=1,2$. This shows that $T\oplus T$ is topologically transitive (and $T$ is mixing if $(n_k)_k=(n)_n$).
\end{proof}

By the way, since every weakly mixing operator satisfies the conditions of the Hypercyclicity Criterion, see \cite{BePe99}, the above version is only formally weaker. 

Yet, it is better adapted to our present needs: it allows us to characterize when a weighted backward shift is hypercyclic (or, equivalently, weakly mixing) and when it is mixing, under the assumption that the tree is unrooted. We start again with the unweighted shift.

\begin{theorem}\label{t-charHCunroot}
Let $(V,E)$ be an unrooted directed tree and $\mu$ a weight. Let $X=\ell^p(V,\mu)$, $1\leq p<\infty$, or $X=c_0(V,\mu)$, and suppose that the backward shift $B$ is a bounded operator on $X$. 

{\rm (a)} The following assertions are equivalent:
\begin{enumerate}[{\rm (i)}]
    \item $B$ is hypercyclic;
    \item $B$ is weakly mixing;
    \item there is an increasing sequence $(n_k)_k$ of positive integers such that, for each $v\in V$, we have as $k\to\infty$,
		\begin{align*}
		\left.
		\begin{aligned} &\inf_{u\in \Chi^{n_k}(v)} |\mu_u| \to 0\quad \text{and}\\ 
		&\min\Big(|\mu_{\prt^{n_k}(v)}|, \inf_{u\in \Chi^{n_k}(\prt^{n_k}(v))} |\mu_u|\Big) \to 0
		\end{aligned}
		\right\} &\text{ if $X=\ell^1(V,\mu)$;}
		\\[2mm]		
	   \left.
		\begin{aligned} &\sum_{u\in \Chi^{n_k}(v)} \frac{1}{|\mu_u|^{p^*}} \to \infty\quad \text{and}\\
		&\frac{1}{|\mu_{\prt^{n_k}(v)}|^{p^*}}+ \sum_{u\in \Chi^{n_k}(\prt^{n_k}(v))} \frac{1}{|\mu_u|^{p^*}} \to \infty
		\end{aligned}
		\right\} &\text{ if $X=\ell^p(V,\mu)$, $1<p<\infty$};
		\\[2mm] 
		\left.
		\begin{aligned} &\sum_{u\in \Chi^{n_k}(v)} \frac{1}{|\mu_u|} \to \infty\quad \text{and}\\ 
		&\frac{1}{|\mu_{\prt^{n_k}(v)}|}+ \sum_{u\in \Chi^{n_k}(\prt^{n_k}(v))} \frac{1}{|\mu_u|} \to \infty
		\end{aligned}
		\right\} &\text{ if $X=c_0(V,\mu)$.}
   	    \end{align*}
\end{enumerate}

{\rm (b)} $B$ is mixing if and only if condition {\rm (iii)} holds for the full sequence $(n_k)_k=(n)_n$.
\end{theorem}

\begin{proof}
(a) The proof follows the same overall strategy as that of Theorem \ref{t-charHCroot}, but it is technically more demanding. As in that proof, we will only consider the case when $X=\ell^p(V)$, $1\leq p<\infty$.

(iii) $\Rightarrow$ (ii). This time we apply the Hypercyclicity Criterion in the form of Proposition \ref{p-criterion}. Let $X_0=Y_0=\vect\{e_v: v\in V\}$.

Since the first part of the characterizing condition in (iii) is identical to the condition in Theorem \ref{t-charHCroot}(iii) we can again define maps $R_{n_k}:Y_0\to \ell^p(V,\mu)$, $k\geq 1$, such that, for all $g\in Y_0$,
\[
R_{n_k}g\to 0\text{ as $k\to\infty$} \quad\text{and}\quad B^{n_k}R_{n_k} g=g, k\geq 1.
\]

Next, for $v\in V$ and $k\geq 1$, let
\[
M_{v,k} := \frac{1}{|\mu_{\prt^{n_k}(v)}|^{p^*}}+ \sum_{u\in \Chi^{n_k}(\prt^{n_k}(v))} \frac{1}{|\mu_u|^{p^*}};
\]
note that $M_{v,k}<\infty$ by the continuity of $B^{n_k}$. Then either
\begin{equation}\label{eq-a}
\frac{1}{|\mu_{\prt^{n_k}(v)}|^{p^*}}\geq \frac{M_{v,k}}{2}\tag{$\alpha$}
\end{equation}
or 
\begin{equation}\label{eq-b}
\sum_{u\in \Chi^{n_k}(\prt^{n_k}(v))} \frac{1}{|\mu_u|^{p^*}}>\frac{M_{v,k}}{2}.\tag{$\beta$}
\end{equation}

In the case of $(\alpha)$, let us define
\[
f_{v,n_k}=e_v.
\]
Then we have that $B^{n_k}f_{v,n_k} = e_{\prt^{n_k}(v)}$ and thus
\[
\|B^{n_k}f_{v,n_k}\| =|\mu_{\prt^{n_k}(v)}|\leq \Big( \frac{2}{M_{v,k}}\Big)^{1/p^*}.
\]

In the case of $(\beta)$ there is by Lemma \ref{l-revhol}, applied to $J=\Chi^{n_k}(\prt^{n_k}(v))$, some $h_{v,n_k}\in \mathbb{K}^V$ of support in $\Chi^{n_k}(\prt^{n_k}(v))$ such that
\begin{equation}\label{eq-two}
\sum_{u\in \Chi^{n_k}(\prt^{n_k}(v))}|h_{v,n_k}(u)|=1
\end{equation}
and
\[
\Big(\sum_{u\in \Chi^{n_k}(\prt^{n_k}(v))}|h_{v,n_k}(u)\mu_u|^p\Big)^{1/p} \leq \Big( \frac{2}{M_{v,k}}\Big)^{1/p^*}.
\]
(For $X=c_0(V,\mu)$, we choose $h_{v,n_k}$ to be in addition of finite support so that $h_{v,n_k}\in c_0(V,\mu)$; see the final assertion in Lemma \ref{l-revhol}.) We may suppose that $h_{v,n_k}(w)\leq 0$ for all $w\in V$. We then define
\[
f_{v,n_k}=e_v+h_{v,n_k},
\]
so that
\[
\|f_{v,n_k} - e_v\| \leq \Big( \frac{2}{M_{v,k}}\Big)^{1/p^*}
\]
and, by \eqref{eq-two},
\[
B^{n_k}f_{v,n_k} = \Big(1+\sum_{u\in \Chi^{n_k}(\prt^{n_k}(v))}h_{v,n_k}(u)\Big)e_{\prt^{n_k}(v)} =0.
\]

In both cases we therefore have that 
\[
\|f_{v,n_k} - e_v\| \leq \Big( \frac{2}{M_{v,k}}\Big)^{1/p^*}\quad\text{and}\quad \|B^{n_k}f_{v,n_k}\| \leq \Big( \frac{2}{M_{v,k}}\Big)^{1/p^*}.
\]
We can then define the maps $I_{n_k}:X_0\rightarrow \ell^p(V,\mu)$, $k\geq 1$, by setting
\[
I_{n_k}e_v=f_{v,n_k}, v\in V
\]
and extending linearly to $X_0$. Then we have for all $f\in X_0$
\[
I_{n_k}f \to f\quad\text{and}\quad B^{n_k}I_{n_k}f \to 0
\]
as $k\to\infty$. Thus all the hypotheses of Proposition \ref{p-criterion} are satisfied, so that $T$ is weakly mixing.

(i) $\Rightarrow$ (iii).  Let $F\subset V$ be finite, $N\geq 1$ and $\eps >0$. We choose some $\delta>0$ with $\delta\leq\eps$ such that, for any $h=(h(v))_{v\in V}\in \ell^p(V, \mu)$, $\|h-\sum_{v\in F}e_v\|<\delta$ implies that $|h(v)|\geq \tfrac{1}{2}$ for all $v\in F$. We also form a subset $G$ of $F$ by picking exactly one vertex from $F$ per generation; that is, if $u\neq v$ in $G$ then $u\nsim v$, and for each $v\in F$ there is some $u\in G$ with $u\sim v$. 

By topological transitivity of $B$ there is now some $f\in \ell^p(V, \mu)$ and some $n\geq N$ such that
\begin{equation}\label{eq-tt}
\Big\|f-\sum_{v\in G}e_v\Big\|< {\eps}\quad\text{and}\text \quad \Big\|B^nf-\sum_{v\in F}e_v\Big\|<\delta;
\end{equation}
by making $n$ bigger, if necessary, we may assume that 
\begin{equation}\label{eq-tt2}
\Chi^n(F)\cap G=\varnothing\quad\text{and}\text \quad \prt^n(G)\cap F=\varnothing
\end{equation}
and 
\begin{equation}\label{eq-tt3}
u,v\in F, u\sim v \Longrightarrow \prt^n(u)=\prt^n(v).
\end{equation}

We claim that this implies that, for all $v\in F$, 
\begin{equation}\label{eq-cond1}
\sum_{u\in \Chi^{n}(v)} \frac{1}{|\mu_u|^{p^*}} > \frac{1}{(2\eps)^{p^*}}
\end{equation}
and
\begin{equation}\label{eq-cond2}
\frac{1}{|\mu_{\prt^{n}(v)}|^{p^*}}+ \sum_{u\in \Chi^{n}(\prt^{n}(v))} \frac{1}{|\mu_u|^{p^*}} \geq \frac{1}{(2\eps)^{p^*}}.
\end{equation}

Once this is shown, it suffices to apply the result to an increasing sequence $(F_k)_k$ of finite subsets of $V$ with $\bigcup_{k\geq 1}F_k=V$, to an increasing sequence $(N_k)_k$ of positive integers, and to a positive sequence $(\eps_k)_k$ tending to 0 in order to deduce the characterizing condition in (iii).

Thus, let us first show \eqref{eq-cond1}. To this end, let $v\in F$. We restrict $f$ to $\Chi^{n}(v)$; then \eqref{eq-tt} with \eqref{eq-tt2} imply that $\|f\chi_{\Chi^n(v)}\|< \eps$. We can now argue exactly as in the proof of Theorem \ref{t-charHCroot} to deduce, using \eqref{eq-tt}, that
\[
\Big(\sum_{u\in \Chi^{n}(v)} \frac{1}{|\mu_u|^{p^*}}\Big)^{-1/p^*} < 2\eps,  
\]
and hence \eqref{eq-cond1}.

As for \eqref{eq-cond2}, it follows from \eqref{eq-tt3} that it suffices, by definition of $G$, to obtain it only for all the vertices in $G$. Thus let $v\in G$. We define
\[
g= f\chi_{\Chi^n(\prt^n(v))} - e_v.
\]
It follows from \eqref{eq-tt} that
\begin{equation}\label{eq-gp}
\Big(\sum_{u\in \Chi^n(\prt^n(v))} |g(u)\mu_u| ^p\Big)^{1/p} = \|g\| <\eps;
\end{equation}
notice that no other $u\in G$, $u\neq v$, belongs to $\Chi^n(\prt^n(v))$. In addition we have that
\[
(B^n f)(\prt^n(v)) = \sum_{u\in \Chi^n(\prt^n(v))} f(u) = 1+\sum_{u\in \Chi^n(\prt^n(v))} g(u).
\]
By \eqref{eq-tt} with \eqref{eq-tt2} this gives us that
\begin{equation}\label{eq-gu}
|\mu_{\prt^n(v)}|\Big|1+\sum_{u\in \Chi^n(\prt^n(v))} g(u) \Big| <\delta\leq \eps.
\end{equation}

We can now assume that $|\mu_{\prt^{n_k}(v)}|> 2\eps$; for if $|\mu_{\prt^{n}(v)}|\leq 2\eps$, then \eqref{eq-cond2} holds trivially. But then \eqref{eq-gu} implies that
\[
\Big|1+\sum_{u\in \Chi^n(\prt^n(v))} g(u) \Big| < \frac{1}{2},
\]
which gives us that
\[
c_v:= \sum_{u\in \Chi^n(\prt^n(v))} |g(u)|\geq \Big|\sum_{u\in \Chi^n(\prt^n(v))} g(u) \Big| > \frac{1}{2}.
\]
Then $\sum_{u\in \Chi^n(\prt^n(v))}\frac{1}{c_v} |g(u)|=1$. Choosing, in Lemma \ref{l-revhol}, $J=\Chi^n(\prt^n(v))$ and for $x$ the restriction of $\frac{1}{c_v}|g|$ to $J$ we obtain with \eqref{eq-gp} that
\[
\Big(\sum_{u\in \Chi^{n}(\prt^{n}(v))} \frac{1}{|\mu_u|^{p^*}}\Big)^{-1/p^*} \leq \frac{1}{c_v} \Big(\sum_{u\in \Chi^{n}(\prt^{n}(v))} |g(u)\mu_u| ^p\Big)^{1/p} < 2\eps,
\]
so that \eqref{eq-cond2} also holds in this case. This completes the proof of (a).

(b) This can be shown exactly as in the proof of (a) noting the final assertion in Proposition \ref{p-criterion} and the definition of the mixing property.
\end{proof}

We shift again the weights from the space to the operator, arriving easily at the following.

\begin{theorem}\label{t-charHCunrootalt}
Let $(V,E)$ be an unrooted directed tree and $\lambda$ a weight. Let $X=\ell^p(V)$, $1\leq p<\infty$, or $X=c_0(V)$, and suppose that the weighted backward shift $B_\lambda$ is a bounded operator on $X$. 

{\rm (a)} The following assertions are equivalent:
\begin{enumerate}[{\rm (i)}]
    \item $B_\lambda$ is hypercyclic;
    \item $B_\lambda$ is weakly mixing;
    \item there is an increasing sequence $(n_k)_k$ of positive integers such that, for each $v\in V$, we have as $k\to\infty$,
		\begin{align*}
		\left.
		\begin{aligned} &\sup_{u\in \Chi^{n_k}(v)} |\lambda(v\to u)| \to \infty\quad \text{and}\\ 
		&\frac{\max(1,\sup_{u\in \Chi^{n_k}(\prt^{n_k}(v))} |\lambda(\prt^{n_k}(u)\to u)|)}{|\lambda(\prt^{n_k}(v)\to v)|}\to \infty
		\end{aligned}
		\right\} &\text{ if $X=\ell^1(V)$;}
		\\[2mm]		
	   \left.
		\begin{aligned} &\sum_{u\in \Chi^{n_k}(v)} |\lambda(v\to u)|^{p^*} \to \infty\quad \text{and}\\ 
		&\frac{1+ \sum_{u\in \Chi^{n_k}(\prt^{n_k}(v))} |\lambda(\prt^{n_k}(u)\to u)|^{p^*} }{|\lambda(\prt^{n_k}(v)\to v)|^{p^*} }\to \infty
		\end{aligned}
		\right\} &\text{ if $X=\ell^p(V)$, $1<p<\infty$;}
		\\[2mm]
		\left.
		\begin{aligned} &\sum_{u\in \Chi^{n_k}(v)} |\lambda(v\to u)| \to \infty\quad \text{and}\\ 
		&\frac{1+ \sum_{u\in \Chi^{n_k}(\prt^{n_k}(v))} |\lambda(\prt^{n_k}(u)\to u)|}{|\lambda(\prt^{n_k}(v)\to v)|}\to \infty
		\end{aligned}
		\right\} &\text{ if $X=c_0(V)$.}
   	    \end{align*}
\end{enumerate}

{\rm (b)} $B_\lambda$ is mixing if and only if condition {\rm (iii)} holds for the full sequence $(n_k)_k=(n)_n$.
\end{theorem}

It is interesting to consider again as our first special case the Rolewicz operators, see Example \ref{ex-rol}. It is easy to see that these operators are never hypercyclic on the branchless unrooted tree $\mathbb{Z}$. But they may well be hypercyclic if the tree has some branches. And once more the space $\ell^1(V)$ plays a special role.

\begin{corollary}\label{c-rolunrooted}
Let $(V,E)$ be an unrooted directed tree. Let $\lambda\in\KK$ be a non-zero number and $\lambda B$ the corresponding Rolewicz operator.

{\rm (a)} The operator $\lambda B$ is never a hypercyclic operator on $\ell^1(V)$.

{\rm (b)} Let $X=\ell^p(V)$, $1<p<\infty$, or $X=c_0(V)$, and let $\lambda B$ be bounded on $X$. Then $\lambda B$ is hypercyclic on $X$ if and only if 
\begin{enumerate}[{\rm (i)}]
\item there is an increasing sequence $(n_k)_k$ of positive integers such that, for each $v\in V$,
\[
|\Chi^{n_k}(v)|\,|\lambda|^{n_kp^*} \to\infty 
\]
as $k\to\infty$,
\item if the tree has a free left end, then $|\lambda|<1$.
\end{enumerate} 
And $\lambda B$ is mixing on $X$ if and only if 
\begin{enumerate}[{\rm (i)}]
\item for each $v\in V$,
\[
|\Chi^{n}(v)|\,|\lambda|^{np^*} \to\infty 
\]
as $n\to\infty$,
\item if the tree has a free left end, then $|\lambda|<1$.
\end{enumerate} 
Here, $p^*=1$ for $X=c_0(V)$.
\end{corollary}

\begin{proof} Assertion (a) is obvious from Theorem \ref{t-charHCunrootalt}. For (b), by that theorem, $\lambda B$ is hypercyclic if and only (i) holds and 
\begin{enumerate}
\item[{\rm (ii')}] \text{for each $v\in V$}, $\frac{1}{|\lambda|^{{n_k}p^*}} + |\Chi^{n_k}(\prt^{n_k}(v))|\to \infty$
\end{enumerate} 
as $k\to\infty$. Now, if the tree does not have a free left end then, for any $v\in V$, $|\Chi^{n}(\prt^{n}(v))|\to \infty$ as $n\to\infty$; note that by (i) the tree cannot have leaves. Hence condition (ii') is satisfied. On the other hand, if the tree has a free left end then there are $v\in V$ with $\Chi^{n}(\prt^{n}(v))=\{v\}$ for all $n\geq 0$; hence condition (ii') is equivalent to demanding that $|\lambda|<1$. The proof for the mixing property is the same.
\end{proof}

If, in particular, every vertex in an unrooted directed tree $(V,E)$ has exactly $N\geq 2$ children, then the Rolewicz operator $\lambda B$ is hypercyclic on $X$ (and then mixing) if and only if 
\begin{align*}
|\lambda|>N^{-1/p^*} & \text{ if $X=\ell^p(V)$, $1<p<\infty$};\\
|\lambda|>N^{-1} & \text{ if $X=c_0(V)$}.
\end{align*}
As already mentioned, if $N=1$, that is, when $V=\mathbb{Z}$, no Rolewicz operator is hypercyclic on any of the spaces considered here.

Applying Theorem \ref{t-charHCunrootalt} to weighted backward shifts on symmetric unrooted directed trees and with a symmetric weight, we get the following corollary.

\begin{corollary}
Let $(V,E)$ be a symmetric, unrooted directed tree and $\lambda=(\lambda_n)_{n\in \ZZ}$ a symmetric weight. Let $X=\ell^p(V)$, $1\leq p<\infty$, or $X=c_0(V)$, and suppose that $B_{\lambda}$ is bounded on $X$.

{\rm (a)} The operator $B_{\lambda}$ is hypercyclic on $X$ if and only if there is an increasing sequence $(n_k)_k$ of positive integers such that, for each $j\in \ZZ$, we have as $k\rightarrow \infty$,
\begin{align*}
		\lambda_{j-n_k+1}\cdots \lambda_{j}\rightarrow 0 \quad \text{and}\quad|\lambda_{j+1}\cdots \lambda_{j+n_k}|\rightarrow \infty\quad
		 &\text{ if $X=\ell^1(V)$;}
		\\[2mm]		
	   \left.
		\begin{aligned} &\min \left( \frac{1}{\gamma_{j-n_k}\cdots \gamma_{j-1}}, |\lambda_{j-n_k+1}\cdots \lambda_{j}|^{p^*}\right)\rightarrow 0 \text{ and}\\
		&\gamma_j\cdots \gamma_{j+n_k-1}|\lambda_{j+1}\cdots \lambda_{j+n_k}|^{p^*}\rightarrow \infty
		\end{aligned}
		\right\} &\text{ if $X=\ell^p(V)$, $1<p<\infty$, or $X=c_0(V)$};
\end{align*}
where $p^\ast=1$ for $X=c_0(V)$.

{\rm (b)} The operator $B_{\lambda}$ is mixing on $X$ if and only if, as $n\rightarrow \infty$,
\begin{align*}
		\lambda_{-n+1}\cdots \lambda_{0}\rightarrow 0 \quad \text{and}\quad|\lambda_{1}\cdots \lambda_{n}|\rightarrow \infty\quad
		 &\text{ if $X=\ell^1(V)$;}
		\\[2mm]		
	   \left.
		\begin{aligned} &\min \left( \frac{1}{\gamma_{-n}\cdots \gamma_{-1}}, |\lambda_{-n+1}\cdots \lambda_{0}|^{p^*}\right)\rightarrow 0 \text{ and}\\
		&\gamma_0\cdots \gamma_{n-1}|\lambda_{1}\cdots \lambda_{n}|^{p^*}\rightarrow \infty
		\end{aligned}
		\right\} &\text{ if $X=\ell^p(V)$, $1<p<\infty$, or $X=c_0(V)$};
\end{align*}
where $p^\ast=1$ for $X=c_0(V)$.
\end{corollary}

Again, the symmetry of the directed tree is not needed if $X=\ell^1(V)$.

We next have an interesting consequence of Theorem \ref{t-charHCunrootalt} for general weighted shifts; an analogous result holds for the mixing property.

\begin{corollary}\label{c-simplifiedunrooted}
Let $(V,E)$ be an unrooted directed tree and $\lambda$ a weight. Let $X=\ell^p(V)$, $1\leq p<\infty$, or $X=c_0(V)$, and suppose that the weighted backward shift $B_\lambda$ is a bounded operator on $X$. 

{\rm (a)} Suppose that there is an increasing sequence $(n_k)_k$ of positive integers such that, for each $v\in V$,
\[
\lambda(\prt^{n_k}(v)\to v)\to 0
\]
and 	
\begin{align*}
		\sup_{u\in \Chi^{n_k}(v)} |\lambda(v\to u)| \to \infty, &\text{ if $X=\ell^1(V)$,}\\
    \sum_{u\in \Chi^{n_k}(v)} |\lambda(v\to u)|^{p^*} \to \infty, &\text{ if $X=\ell^p(V)$, $1<p<\infty$,}\\
    \sum_{u\in \Chi^{n_k}(v)} |\lambda(v\to u)| \to \infty, &\text{ if $X=c_0(V)$,}
\end{align*}
as $k\to\infty$, then $B_\lambda$ is hypercyclic. 

{\rm (b)} If $(V,E)$ has a free left end, then the above conditions characterize when $B_\lambda$ is hypercyclic.
\end{corollary}

\begin{proof}
Part (a) is obvious. As for (b), we only consider the case when $X=\ell^p(V)$, $1<p<\infty$. Let $v\in V$. Assuming  the second condition in Theorem \ref{t-charHCunrootalt}(iii) we need to show that 
\begin{equation}\label{eq-lam}
\lambda(\prt^{n_k}(v)\rightarrow v)\rightarrow 0
\end{equation}
as $k\rightarrow \infty$.
The generation of $v$ has a common ancestor, that is, there is some $w\in V$ and some $N\geq 0$ such that, for all $u\in V$ with $u\sim v$,
\[
w=\prt^N(u),
\]
and we have for all $n\geq N$ that
\[
\Chi^n(\prt^n(v)) = \Chi^N(w).
\]
We then deduce with \eqref{eq-lambda} that, for all $k$ with $n_k\geq N$,
\begin{align*}
\sum_{u\in \Chi^{n_k}(\prt^{n_k}(v))} |\lambda(\prt^{n_k}(u)\to u)|^{p^*}&= |\lambda(\prt^{n_k-N}(w)\to w)|^{p^*} \sum_{u\in \Chi^{N}(w)} |\lambda(w\to u)|^{p^*}\\
&= |\lambda(\prt^{n_k}(v)\to v)|^{p^*}\frac{\sum_{u\in \Chi^{N}(w)} |\lambda(w\to u)|^{p^*}}{|\lambda(w\to v)|^{p^*}}.
\end{align*}
This, together with the second condition in Theorem \ref{t-charHCunrootalt}(iii), implies \eqref{eq-lam}.
\end{proof}

We remark that assertion (b) in the corollary does not hold without the assumption that $(V,E)$ has a free left end. Indeed, by the discussion following Corollary \ref{c-rolunrooted}, the unweighted backward shift $B$ is hypercyclic on $c_0(V)$, say, for the unrooted directed tree in which each vertex has exactly two children, while the first condition in (a) is not satisfied.

Finally, Corollary \ref{c-simplifiedunrooted} and the proof of Corollary \ref{c-charHCrootalt} imply the following sufficient condition for hypercyclicity on particular unrooted trees; an analogous result holds for the mixing property. Recall the definition of $\Lambda_{v,n}$ given in \eqref{eq-Lambda}.

\begin{corollary}\label{c-charHCunrootalt}
Let $(V,E)$ be an unrooted directed tree and $\lambda$ a weight. Let  $B_\lambda$ be a bounded operator on $\ell^p(V)$, $1\leq p<\infty$. Suppose that every vertex has at most finitely many children. If there is an increasing sequence $(n_k)_k$ of positive integers such that, for each $v\in V$,
\[
\lambda(\prt^{n_k}(v)\to v)\to 0
\]
and
\[
\frac{1}{|\Chi^{n_k}(v)|^{1/p}}\Lambda_{v,n_k}\to \infty
\]
as $k\to\infty$, then $B_\lambda$ is hypercyclic.
\end{corollary}

Note that Martínez-Avenda\~{n}o \cite[Theorem 6.1]{Mar17} obtains yet another sufficient condition, see also Remark \ref{r-mart_suff}.

\section{Existence of hypercyclic weighted shifts}\label{s-exist}

We finally address the question mentioned at the end of Section \ref{s-root}: Given any (rooted or unrooted) directed tree, does it support a hypercyclic or even mixing weighted backward shift? Of course, the tree has to be leafless, see Remark \ref{r-leaf}, but this is the only geometric restriction. 

\begin{theorem}\label{HCmix-HCnmix}
Let $(V,E)$ be a leafless directed tree. Let $X=\ell^p(V)$, $1\leq p<\infty$, or $X=c_0(V)$.

{\rm (a)} There exists a mixing weighted backward shift on $X$.

{\rm (b)} There exists a hypercyclic non-mixing weighted backward shift on $X$.
\end{theorem}

\begin{proof}
We will do the proof for the space $X=\ell^p(V)$, $1<p<\infty$, on an unrooted leafless directed tree since the remaining cases are simpler. Fix a vertex $v_0 \in V$ and enumerate the generations with respect to $v_0$. 

(a) For any vertex $v\in \gen_n$, $n\leq -1$, we choose weights $\lambda_u$, $u\in \Chi(v)$, such that $\sum_{u\in \Chi(v)} |\lambda_u|^{p^*}\leq \frac{1}{2}$. On the other hand, for any vertex $v\in \gen_n$, $n\geq 0$, we fix a child $u_v$ of $v$ and put $\lambda_{u_v}=2$, while on the remainder of $\Chi(v)$ we choose weights such that $\sum_{u\in \Chi(v)} |\lambda_u|^{p^*}\leq 3^{p^*}$. Then $B_{\lambda}$ is bounded on $\ell^p(V)$ by Proposition \ref{p-Bbounded}, and it is easily seen by the mixing analogue of Corollary \ref{c-simplifiedunrooted} that $B_\lambda$ is mixing; the crucial point is that, starting from any $v\in V$, one can choose a branch going to infinity whose weights eventually all have the value 2.

(b) We start as in (a), which gives us $\lambda_u$ for any $u\in \gen_n$, $n\leq 0$, such that
\begin{equation}\label{eq-neg}
\sum_{u\in \Chi(v)} |\lambda_u|^{p^*}\leq \frac{1}{2},\quad v\in \bigcup_{n\leq -1} \gen_n.
\end{equation}
We construct the remaining weights recursively. For this, fix an enumeration $v_n$, $n\geq 1$, of $V\setminus\{v_0\}$. Let $N_n\in \mathbb{Z}$ be such that $v_n\in \gen_{N_n}$ for $n \geq 0$; note that $N_0=0$.

We now claim that there are sequences $(m_k)_{k\geq 0}$ and $(n_k)_{k\geq 0}$ of non-negative integers and weights $\lambda_v$, $v\in \bigcup_{n\geq 1} \gen_n$, such that, for any $k\geq 1$,
\begin{equation}\label{eq-mj}
m_k > n_{k-1} + \max(N_0,N_1,\ldots,N_{k}),
\end{equation}
\begin{equation}\label{eq-nj}
n_k > m_k,
\end{equation}
\begin{equation}\label{eq-nmix}
\sum_{u\in \Chi^{m_k}(v_0)} |\lambda(v_0\to u)|^{p^*}\leq 1,
\end{equation}
\begin{equation}\label{eq-hyp}
\sum_{u\in \Chi^{n_k}(v_j)} |\lambda(v_j\to u)|^{p^*}>2^k, j=0,\ldots,k; 
\end{equation}
moreover, for any $v\in \bigcup_{0\leq n < r_k} \gen_n$, $k\geq 0$,
\begin{equation}\label{eq-bdd}
\sum_{u\in \Chi(v)} |\lambda_u|^{p^*}\leq 3^{p^*},
\end{equation}
where $r_k:=n_k+\max(N_0,N_1,\ldots,N_k)$.

Let us first show that this implies the claim. Since $r_k\to\infty$ as $k\to\infty$, $\lambda_v$ is defined for all $v\in V$, and \eqref{eq-neg} and \eqref{eq-bdd} imply that the weighted backward shift $B_\lambda$ is bounded on $\ell^p(V)$. Now, it follows from \eqref{eq-neg} and \eqref{eq-hyp} with Corollary \ref{c-simplifiedunrooted} that $B_\lambda$ is hypercyclic, and \eqref{eq-nmix} shows with Theorem \ref{t-charHCunrootalt} that $B_\lambda$ is not mixing.

Thus it suffices to show recursively that \eqref{eq-mj}-\eqref{eq-bdd} may be satisfied for all $k\geq 0$. For $k=0$, simply take $m_0=n_0=0$; note that $r_0=0$. 

Now assume that $m_0,m_1,\ldots,m_{k}$, $n_0,\ldots, n_k$, and $\lambda_u$, $u\in \bigcup_{0< n \leq r_k} \gen_n$ have been defined for some $k\geq 0$ so that \eqref{eq-mj}-\eqref{eq-bdd} hold.
 
We choose
\[
m_{k+1} > n_k+\max(N_0,N_1,\ldots,N_{k+1}),
\]
hence $m_{k+1} >r_k$, and we define the weights on $\bigcup_{r_k<n\leq m_{k+1}}\gen_{n}$ in such a way that
\[
\sum_{u\in \Chi(v)} |\lambda_u|^{p^*}\leq 3^{p^*}, \quad v\in \bigcup_{r_k\leq n< m_{k+1}}\gen_{n}
\]
and
\[
\sum_{u\in \Chi^{m_{k+1}}(v_0)} |\lambda(v_0\to u)|^{p^*}\leq 1.
\]
Let 
\begin{equation}\label{eq-alpha}
\alpha_{j}:= \sum_{u\in \Chi^{m_{k+1}-N_j}(v_j)} |\lambda(v_j\to u)|^{p^*}, \quad j=0,\ldots,k+1;
\end{equation}
note that $m_{k+1}-N_j>0$ and $\Chi^{m_{k+1}-N_j}(v_j)\subset \gen_{m_{k+1}}$, so the right-hand side is well defined. Then choose
\[
n_{k+1}> m_{k+1}
\]
such that
\begin{equation}\label{eq-alphabis}
2^{(n_{k+1}-m_{k+1}+N_j)p^*}\alpha_{j} > 2^{k+1}, \quad j=0,\ldots,k+1.
\end{equation}
Write
\[
s_{k+1}= r_{k+1}-m_{k+1}=n_{k+1}- m_{k+1}+\max(N_0,N_1,\ldots,N_{k+1})> 0.
\]

Now consider $u\in \bigcup_{j=0}^{k+1}\Chi^{m_{k+1}-N_j}(v_j)\subset \gen_{m_{k+1}}$. It is not excluded that $u$ appears here as a descendant of several $v_j$, $0\leq j\leq k+1$. Then fix a lineage $u_1\in \Chi(u), u_2\in \Chi(u_1), \ldots, u_{s_{k+1}}\in \Chi(u_{s_{k+1}-1})$ and define $\lambda_{u_l}=2$ for $l=1,\ldots, s_{k+1}$. Note that $u_{s_{k+1}}\in \gen_{r_{k+1}}$.

It follows from our construction that any $v\in \bigcup_{m_{k+1}\leq n < r_{k+1}} \gen_n$ has at most one child for which we have defined the weight to be 2. Thus we can now choose the undefined weights in $\bigcup_{m_{k+1}< n \leq r_{k+1}} \gen_n$ in such a way that, for all $v\in \bigcup_{m_{k+1}\leq n < r_{k+1}} \gen_n$,
\[
\sum_{u\in \Chi(v)} |\lambda_u|^{p^*}\leq 3^{p^*}.
\]
It follows from \eqref{eq-alpha}, \eqref{eq-alphabis} with \eqref{eq-lambda} and the fact that $s_{k+1}\geq n_{k+1}-m_{k+1}+N_j$ for $j=0,\ldots, k+1$ that
\[
\sum_{u\in \Chi^{n_{k+1}}(v_j)} |\lambda(v_j\to u)|^{p^*}>2^{k+1},\quad j=0,\ldots, k+1.
\]

Altogether we have shown that \eqref{eq-mj}-\eqref{eq-bdd} are satisfied for $k+1$. 

This finishes the proof.

\end{proof}

\section{Outlook} Of course, this is not the end of a story but, as we hope, only a beginning. In future work we intend to study the question of chaos for the operators considered in this paper. Another, and in view of the results in \cite{BaRu15}, \cite{BoGE18} and \cite{Gro19} probably rather technical, problem is that of characterizing frequent hypercyclicity. Let us repeat our hope alluded to in the introduction that weighted shifts on directed trees might provide new and easier counter-examples in linear dynamics; the additional flexibility provided by the geometry of the tree should turn out to be helpful.

Martínez-Avenda\~{n}o and Rivera-Guasco \cite{MaRi20} have already looked at the hypercyclicity of the backward shift on Lipschitz spaces of directed trees. More generally, one would want a description of the dynamics of weighted backward shifts on very general Banach or Fréchet spaces over a directed tree. 

Let us finish with a specific problem. It is yet another well-known result of Salas \cite{Salas} that for any unilateral weighted backward shift $B_\lambda$ on $\ell^p(\NN)$, $1\leq p<\infty$, or $c_0(\NN)$, the operator $I+B_\lambda$ is hypercyclic. This was generalized by Grivaux and Shkarin; in particular, if $T$ is an operator on a separable Banach space $X$ for which $\bigcup_{n\geq 0} (\text{ker}(T^n)\cap \text{ran}(T^n))$ is dense in $X$, then $I+T$ is mixing, see \cite[Section 2.1]{BM09}. Since, for weighted backward shifts on rooted directed trees, the latter set contains all sequences of finite support, we have the following.

\begin{corollary} 
Let $(V,E)$ be a rooted directed tree. Then, for any weighted backward shift $B_\lambda$ on any of the spaces $\ell^p(V)$, $1\leq p<\infty$, or $X=c_0(V)$, the operator $I+B_\lambda$ is mixing.
\end{corollary} 

On the other hand, the orbits of the operator $I+2B$ on $\ell^2(\ZZ)$ are increasing in norm, so that the operator cannot be hypercyclic, while Salas \cite{Sal07} has found some perturbations of the identity by bilateral weighted shifts that are hypercyclic. 

\begin{problem}
(a) (Salas \cite{Sal07}) Characterize the weighted backward shifts $B_\lambda$ on $\ell^p(\ZZ)$, $1\leq p<\infty$, or $c_0(\ZZ)$ such that $I+B_\lambda$ is hypercyclic.

(b) Do the same for arbitrary unrooted directed trees.
\end{problem}

\textbf{Acknowledgement.} The second author was supported by the Ministry of Science and Higher Education of the Russian Federation, agreement No: 075–15–2019–1619.

The authors wish to thank the referee for his/her numerous helpful comments that greatly improved the presentation of this paper. We are also grateful for a very useful list of additional references that helped us to place our work in a wider context.

\end{document}